\documentclass{amsart}%
\usepackage{amsfonts}
\usepackage{amsmath}
\usepackage{amssymb}
\usepackage{graphicx}%
\newtheorem{theorem}{Theorem}[section]
\theoremstyle{plain}

\newtheorem{lemma}[theorem]{Lemma}
\newtheorem{proposition}[theorem]{Proposition}
\newtheorem{definition}[theorem]{Definition}
\theoremstyle{remark}
\newtheorem{remark}[theorem]{Remark}
\numberwithin{equation}{section}
\begin{document}
\title[Inverse Scattering for the Novikov-Veselov Equation]{Miura Maps and Inverse Scattering for the Novikov-Veselov Equation}
\author{Peter A. Perry}
\thanks{Version of
%TCIMACRO{\TeXButton{today}{\today}}%
%BeginExpansion
\today
%EndExpansion
}
\thanks{Supported in part by NSF\ grants DMS-0710477 and DMS-1208778.}

\begin{abstract}
We use the inverse scattering method to solve the zero-energy Novikov-Veselov
(NV)\ equation for initial data of conductivity type, solving a problem posed
by Lassas, Mueller, Siltanen, and Stahel. We exploit Bogadanov's Miura-type
map which transforms solutions of the modified Novikov-Veselov (mNV) equation
into solutions of the NV\ equation. We show that the Cauchy data of
conductivity type considered by Lassas, Mueller, Siltanen, and Stahel lie in
the range of Bogdanov's Miura-type map, so that it suffices to study the mNV
equation. We solve the mNV equation using the scattering transform associated
to the defocussing Davey-Stewartson II\ equation.

\end{abstract}
\maketitle
\tableofcontents

\section{Introduction}

In this paper we will use inverse scattering methods to solve the
Novikov-Veselov (NV)\ equation, a completely integrable, dispersive nonlinear
equation in two space and one time ($2+1$) dimensions, for the class of
\emph{conductivity type} initial data that we define below. Our results solve
a problem posed by Lassas, Mueller, Siltanen, and Stahel \cite{LMSS:2011} in
their analytical study of the inverse scattering method for the NV\ equation.

Denoting $z=x_{1}+ix_{2}$, $\overline{\partial}=\left(  1/2\right)
(\partial_{x_{1}}+i\partial_{x_{2}})$, $\partial=\left(  1/2\right)
(\partial_{x_{1}}-i\partial_{x_{2}})$, the Cauchy problem for the NV\ equation
is%
\begin{align}
&  q_{t}+\partial^{3}q+\overline{\partial}^{3}q+\frac{3}{4}\partial\left(
q\overline{\partial}^{-1}\partial q\right)  +\frac{3}{4}\overline{\partial
}\left(  q\partial^{-1}\overline{\partial}q\right) \label{eq:NV}\\
\left.  q\right\vert _{t=0}  &  =q_{0}\nonumber
\end{align}
where $q_{0}$ is a real-valued function that vanishes at infinity. Up to
trivial scalings, our equation is the zero-energy ($E=0$) case of the equation%
\begin{align}
q_{t}  &  =4\operatorname{Re}\left(  4\partial^{3}q+\partial(qw)-E\partial
q\right) \label{eq:NV.E}\\
\overline{\partial}w  &  =\partial q\nonumber
\end{align}
studied by Novikov and Veselov in \cite{NV:1984,NV:1986}. If $q$ does not
depend on $y$, the zero-energy NV\ equation (\ref{eq:NV}) reduces (after
rescaling) to the Korteweg-de Vries (KdV)\ equation%
\[
q_{t}=\frac{1}{4}q_{xxx}+6qq_{x}=0.
\]

The Novikov-Veselov equation is one of a hierarchy of dispersive nonlinear
equations in $2+1$ dimensions discovered by Novikov and Veselov
\cite{NV:1984,NV:1986}. In these papers, Novikov and Veselov constructed
explicit solutions from the spectral data associated to a two-dimensional
Schr\"{o}dinger problem at a single energy. Novikov conjectured that the
inverse problem for the two-dimensional Schr\"{o}dinger operator at a fixed
energy should be completely solvable (see the remarks in \cite{Grinevich:2000}%
), and that inverse scattering for the Schr\"{o}dinger equation at a fixed
energy $E$ could be used to solve the NV\ equation at the same energy $E$ by
inverse scattering. In subsequent studies, Grinevich, Grinevich-Manakov, and
Grinevich-Novikov \cite{Grinevich:1986,GN:1986a,GM:1988,GN:1988,GN:1995}
further developed the inverse scattering method and constructed multisoliton
solutions. Independently, Boiti, Leon, Manna, and Pempinelli \cite{BLMP:1987}
proposed the inverse scattering method to solve the NV\ equation at zero
energy with data vanishing at infinity.

It has long been understood that the inverse Schr\"{o}dinger scattering
problem at zero energy poses special challenges (see, for example, the
discussion in Part I of supplement 1 in \cite{GN:1988}, and the comments in
\S 7.3 of \cite{Grinevich:2000}). In particular, the scattering transform for
the Schr\"{o}dinger operator at zero energy is known to be well-behaved
\emph{only} for a special class of potentials, the potentials of
\textquotedblleft conductivity type,\textquotedblright\ defined as follows.

\begin{definition}
A real-valued function $u\in\mathcal{C}_{0}^{\infty}(\mathbb{R}^{2})$ is
called a \emph{potential of conductivity type} if the equation $\left(
-\Delta+q\right)  \psi=0$ admits a unique, strictly positive solution
normalized so that $\psi(z)=1$ in a neighborhood of infinity.
\end{definition}

The class of conductivity type potentials can also be defined for less regular
$q$, but this definition will suffice for the present purpose. The terminology
comes from the connection of the Schr\"{o}dinger inverse problem at zero
energy with Calderon's inverse conductivity problem \cite{Calderon:1980} (see
Astala-Pa\"{\i}v\"{a}rinta \cite{AP:2006} for the solution to Calderon's
inverse problem for $\gamma\in L^{\infty}$, and for references to the
literature). The problem is to reconstruct the conductivity $\gamma$ of a
conducting body $\Omega\subset\mathbb{R}^{2}$ from the Dirichlet to Neumann
map, defined as follows. Let $f\in H^{1/2}(\Omega)$ and let $u\in H^{1}%
(\Omega)$ solve the problem%
\begin{align*}
\nabla\cdot\left(  \gamma\nabla u\right)   &  =0,\\
\left.  u\right\vert _{\partial\Omega}  &  =f.
\end{align*}
This problem has a unique solution for conductivities $\gamma\in L^{\infty
}(\Omega)$ with $\gamma(z)\geq c>0$ for a.e. $z$. The Dirichlet to Neumann map
is the mapping%
\begin{align*}
\Lambda_{\sigma}  &  :H^{1/2}(\partial\Omega)\rightarrow H^{-1/2}%
(\partial\Omega)\\
f  &  \mapsto\left.  \gamma\frac{\partial u}{\partial\nu}\right\vert
_{\partial\Omega}.
\end{align*}
Nachman \cite{Nachman:1995} exploited the fact that $\psi=\gamma^{1/2}u$
solves the Schr\"{o}dinger equation at zero energy where $q=\gamma
^{-1/2}\Delta\left(  \gamma^{1/2}\right)  $. The Schr\"{o}dinger problem also
has a Dirichlet to Neumann map defined by the unique solution of
\begin{align*}
\left(  -\Delta+q\right)  \psi &  =0\\
\left.  \psi\right\vert _{\partial\Omega}  &  =f
\end{align*}
which determines the scattering data for $q$ at zero energy. Note that $q$ is
of conductivity type if we take $\psi=\gamma^{1/2}$ and extend $\psi$ to
$\mathbb{R}^{2}\backslash\Omega$ setting $\psi(z)=1$. Nachman showed that the
scattering transform at zero energy is well-defined \emph{only} when $q$ is of
conductivity type (we give a precise statement below) and used the inverse
scattering transform to reconstruct $q$ from its scattering data.

An important fact is that, under suitable decay and regularity hypotheses, $q$
is a potential of conductivity type if and only if $q$ is a \emph{critical
potential}, i.e., a measurable function $q$ so that the quadratic form
$-\Delta+q$ is well-defined and nonnegative, but the associated
Schr\"{o}dinger operator does not have a positive Green's function. Most
importantly for our purpose, critical potentials have the following property:
if $q$ is a critical potential, then for any nonzero, nonnegative function
$W\in\mathcal{C}_{0}^{\infty}(\mathbb{R}^{2})$ and any $\varepsilon>0$,
$q-\varepsilon W$ is not critical (for a precise statement and references to
the Schr\"{o}dinger operators literature, see the paper of Gesztesy and Zhao
\cite{GZ:1995}). Thus, the set of conductivity-type potentials is nowhere
dense in any reasonable function space! For this reason one expects the direct
and inverse scattering maps for the Schr\"{o}dinger operator at zero energy
not to have good continuity properties as a function of the potential $q$.

Let us describe the direct scattering transform $\mathcal{T}$ and inverse
scattering transform $\mathcal{Q}$ for the Schr\"{o}dinger operator at zero
energy in more detail (see Nachman \cite{Nachman:1995} and Lassas, Mueller,
Siltanen, and Stahel \cite{LMSS:2011} for details and references). To define
the direct scattering map $\mathcal{T}$ on potentials $q\in\mathcal{C}%
_{0}^{\infty}(\mathbb{R}^{2})$, we seek complex geometric optics (CGO)
solutions $\psi=\psi(z,k)$ of
\begin{equation}
\left(  -\Delta+q\right)  \psi=0. \label{eq:SE1}%
\end{equation}
which satisfy the asymptotic condition%
\begin{equation}
\lim_{\left\vert z\right\vert \rightarrow\infty}e^{-ikz}\psi(z,k)=1.
\label{eq:SE2}%
\end{equation}
for a fixed $k\in\mathbb{C}$. Let $m(z,k)=e^{-izk}\psi(z,k)$. Assuming that
the problem (\ref{eq:SE1})-(\ref{eq:SE2}) has a unique solution for all~$k$,
we define the scattering transform $\mathbf{t}=\mathcal{T}q$ via the formula%
\begin{equation}
\mathbf{t}(k)=\int e^{i(\overline{k}\overline{z}+kz)}q(z)m(z,k)~dA(z)
\label{eq:t.rep}%
\end{equation}
where $dA(z)$ is Lebesgue measure on $\mathbb{R}^{2}$. The surprising fact is
that, if $\mathbf{t}$ is well-behaved, the solutions $\psi(z,k)$ may be
recovered from $\mathbf{t}(k)$. This fact leads to an inverse scattering
transform $q=\mathcal{Q}\mathbf{t}$ given by%
\begin{equation}
q(z)=\frac{i}{\pi^{2}}\overline{\partial}_{z}\left(  \int_{\mathbb{C}}%
\frac{\mathbf{t}(k)}{\overline{k}}e^{-i\left(  kz+\overline{k}\overline
{z}\right)  }\overline{m(z,k)}~dA(k)\right)  . \label{eq:q.rep}%
\end{equation}

Boiti, Leon, Manna, and Pempinelli \cite{BLMP:1987}, proposed an inverse
scattering solution to the Novikov-Veselov equation using these maps:%
\begin{equation}
q(t)=\mathcal{Q}\left(  e^{it\left(  \left(  \diamond\right)  ^{3}+\left(
\overline{\diamond}\right)  ^{3}\right)  }\left(  \mathcal{T}q_{0}\right)
\left(  \diamond\right)  \right)  \label{eq:NV.sol}%
\end{equation}
and gave formal arguments to justify it. The maps were further studied by Tsai
in \cite{Tsai:1993}. Lassas, Mueller, Siltanen, and Stahel \cite{LMSS:2011},
building on results of \cite{LMS:2007}, showed that the scattering transforms
are well-defined for certain potentials of conductivity type. For
conductivity-type potentials, Lassas et. al. proved that $\mathcal{T}$ and
$\mathcal{Q}$ are inverses, and that (\ref{eq:NV.sol}) defines a continuous
$L^{p}(\mathcal{R}^{2})$-valued function of $t$ for $p\in(1,2)$. They
conjectured that $q(t)$ is in fact a classical solution of (\ref{eq:NV}) if
$q_{0}$ is a smooth, decreasing, real-valued potential of conductivity type
but were unable to prove that this was the case.

The fact already mentioned, that conductivity-type potentials are a nowhere
dense set in the space of potentials, suggests that studying the NV equation
using the maps $\mathcal{T}$ and $\mathcal{Q}$ is likely to be technically
challenging. The following result of Nachman (\cite{Nachman:1995}, Theorem 3)
makes the difficulty clearer. For given $q$, let $\mathcal{E}_{q}$ be the set
of all $k$ for which the problem (\ref{eq:SE1})-(\ref{eq:SE2}) does \emph{not}
have a unique solution. Let $L_{\rho}^{p}(\mathbb{R}^{2})$ denote the Banach
space of real-valued measurable functions $q$ with $\left\Vert q\right\Vert
_{L_{\rho}^{p}}=\left[  \int\left(  1+\left\vert z\right\vert \right)
^{p\rho}\left\vert q(z)\right\vert ^{p}~dA(z)\right]  ^{1/p}$.

\begin{theorem}
\cite{Nachman:1995} Suppose that $q\in L_{\rho}^{p}(\mathbb{R}^{2})$ for some
$p\in(1,2)$, and $\rho>1$: The following are equivalent:\newline(i) The set
$\mathcal{E}_{q}$ is empty and $\left\vert \mathbf{t}\left(  k\right)
\right\vert \leq C\left\vert k\right\vert ^{\varepsilon}$ for some fixed
$\varepsilon>0$ and all sufficiently small $k$.\newline(ii) There is a
real-valued function $\gamma\in L^{\infty}(\mathbb{R}^{2})$ with
$\gamma(z)\geq c>0$ for a.e. $z$ and a fixed constant $c$ so that
$q=\gamma^{-1/2}\Delta\left(  \gamma^{1/2}\right)  $.
\end{theorem}

One should think of $\gamma$ as $\psi^{2}$ where $\psi$ is the unique
normalized positive solution of $\,(-\Delta+q)\psi=0$ for a potential of
conductivity type. Nachman's result suggests that non-conductivity type
potentials will have singular scattering transforms: in \cite{MPS:2012},
Music, Perry, and Siltanen construct an explicit one-parameter deformation
$\lambda\mapsto q_{\lambda}$ of a conductivity type potentials ($q_{0}$ is of
conductivity type, but $q_{\lambda}$ is not for $\lambda\neq0$) for which the
corresponding family $\lambda\mapsto\mathbf{t}_{\lambda}$ of scattering
transforms has an essential singularity at $\lambda=0$.

We will show that, nonetheless, the formula (\ref{eq:NV.sol}) does yield
classical solutions of the NV\ equation for a much larger class of initial
data than considered in \cite{LMSS:2011}. We achieve this result by
circumventing the scattering maps studied in \cite{LMSS:2011}. Instead, we
exploit Bogdanov's \cite{Bogdanov:1987} observation that the Miura-type map%
\begin{equation}
\mathcal{M}(v)=2\partial v+\left\vert v\right\vert ^{2} \label{eq:Miura2}%
\end{equation}
takes solutions $u$ of the modified Novikov-Veselov (mNV) equation to
solutions $q$ of the NV\ equation with initial data of conductivity type.
Here, the domain of the Miura map is understood to be smooth functions $v$
with $\partial v=\overline{\partial v}$. We will show that the range of
$\mathcal{M}$ contains the conductivity-type potentials studied by Lassas,
Mueller, Siltanen and Stahel.

Thus, to solve the NV\ equation for initial data of conductivity type, it
suffices to solve the mNV equation and use the map $\mathcal{M}$ to obtain a
solution of NV. \ The mNV\ equation is a member of the Davey-Stewartson
II\ (DS\ II)\ hierarchy, so the well-known scattering maps for the DS
II\ hierarchy (see \cite{P} and reference therein) can be used to solve to
solve the Cauchy problem for mNV. We denote by $\mathcal{R}$ and $\mathcal{I}$
respectively the scattering transform and inverse scattering transform
associated to the defocussing Davey-Stewartson II\ equation (see
\S \ref{sec:IR.def} for the definitions). We show in Appendix \ref{app:mNV}
that the function%
\begin{equation}
u(t)=\mathcal{I}\left(  \exp\left(  \left(  \overline{\diamond}^{3}%
-\diamond^{3}\right)  t\right)  \left(  \mathcal{R}u_{0}\right)
(\diamond)\right)  \label{eq:mNV.sol}%
\end{equation}
is a classical solution of the mNV\ equation for initial data $u_{0}%
\in\mathcal{S}(\mathbb{R}^{2})$.

In order to obtain good mapping properties for the solution map $u_{0}\mapsto
u(t)$ defined by (\ref{eq:mNV.sol}), we need local Lipschitz continuity of the
maps $\mathcal{I}$ and $\mathcal{R}$ on spaces that are preserved under the
flow (compare the treatment of the cubic NLS\ in one dimension by Deift-Zhou
\cite{DZ:2003}, and the Sobolev mapping properties for the scattering maps for
NLS proven in \cite{Zhou:1998}). In \cite{P} it was shown that $\mathcal{R}$
and $\mathcal{I}$ are mutually inverse mappings of $H^{1,1}(\mathbb{R}^{2})$
into itself where%
\[
H^{m,n}(\mathbb{R}^{2})=\left\{  u\in L^{2}(\mathbb{R}^{2}):(1-\Delta
)^{m/2}u,~\left(  1+\left\vert ~\cdot~\right\vert \right)  ^{n}u(~\cdot~)\in
L^{2}(\mathbb{R}^{2})\right\}  .
\]

In order to use (\ref{eq:mNV.sol}), we need the following refined mapping
property of $\mathcal{I}$ and $\mathcal{R}$.

\begin{theorem}
\label{thm:Lip}The scattering maps $\mathcal{R}$ and $\mathcal{I}$ restrict to
locally Lipschitz continuous maps%
\begin{align*}
\mathcal{R}  &  :H^{2,1}(\mathbb{R}^{2})\rightarrow H^{1,2}(\mathbb{R}^{2}),\\
\mathcal{I}  &  :H^{1,2}(\mathbb{R}^{2})\rightarrow H^{2,1}(\mathbb{R}^{2}).
\end{align*}

\end{theorem}

Theorem \ref{thm:Lip} immediately implies that the solution formula
(\ref{eq:mNV.sol}) defines a continuous map%
\begin{align*}
H^{2,1}(\mathbb{R}^{2}) &  \rightarrow C\left(  \left[  0,T\right]
;H^{2,1}(\mathbb{R}^{2}\right)  ),\\
t &  \mapsto u(t).
\end{align*}
for any $T>0$. We say that $u$ is a weak solution of the mNV\ equation (see
(\ref{eq:mNV})) on $\left[  0,T\right]  $ if
\begin{equation}
-\left(  \varphi_{t}+\partial^{3}\varphi+\overline{\partial}^{3}%
\varphi,u\right)  +\left(  \varphi,NL(u)\right)  =0\label{eq:mNV.weak}%
\end{equation}
for all $\varphi\in\mathcal{C}_{0}^{\infty}\left(  \mathbb{R}^{2}\times\left[
0,T\right]  \right)  $, where $\left(  ~\cdot~,~\cdot~\right)  $ denotes the
inner product on $L^{2}(\mathbb{R}^{2}\times\left[  0,T\right]  )$. We will
show that (\ref{eq:mNV.sol}) defines a weak solution in this sense and that,
also, the flow (\ref{eq:mNV.sol}) leaves the domain of $\mathcal{M}$
invariant. We will prove:

\begin{theorem}
\label{thm:mNV}For $u_{0}\in\mathcal{S}(\mathbb{R}^{2})$, the solution formula
(\ref{eq:mNV.sol}) gives a classical solution of mNV. Moreover, if $u_{0}\in
H^{2,1}(\mathbb{R}^{2})\cap L^{1}(\mathbb{R}^{2})$, $\partial u_{0}%
=\overline{\partial u_{0}}$, and $\int u_{0}(z)~dA(z)=0$, then $u(t)$ is a
weak solution of mNV and the relations $\left(  \partial u\right)  \left(
~\cdot~,t\right)  =\overline{(\partial u)(~\cdot~,t)}$ and $\int
u(z,t)~dA(z)=0$ hold for all $t$.
\end{theorem}

\begin{remark}
Although it is likely well within the reach of current technology (see e.g.
\cite{KPV:1991} for relevant dispersive estimates), there appear to be no
uniqueness or local well-posedness result for mNV in the literature. Given
such a result, one could conclude from the proof of Theorem \ref{thm:mNV} that
the Cauchy problem for mNV is globally well--posed in $H^{2,1}(\mathbb{R}%
^{2})$.
\end{remark}

Now we can solve the NV\ equation using the solution map for mNV and the Miura
map (\ref{eq:Miura2}). We say that $q$ is a weak solution of the NV equation
on $\left[  0,T\right]  $ if
\begin{equation}
\left(  \varphi_{t}+\partial^{3}\varphi+\overline{\partial}^{3}\varphi
,q\right)  +\frac{3}{4}\left(  \partial\varphi,q\overline{\partial}%
^{-1}\partial q\right)  +\frac{3}{4}\left(  \overline{\partial}\varphi
,q\partial^{-1}\overline{\partial}q\right)  =0 \label{eq:NV.weak}%
\end{equation}
for all $\varphi\in\mathcal{C}_{0}^{\infty}\left(  \mathbb{R}^{2}\times\left[
0,T\right]  \right)  $. Using Theorem \ref{thm:mNV}, we will prove:

\begin{theorem}
\label{thm:NV}Suppose that $q_{0}=2\partial u_{0}+\left\vert u_{0}\right\vert
^{2}$ where $u_{0}\in H^{2,1}(\mathbb{R}^{2})\cap L^{1}(\mathbb{R}^{2})$,
$\partial u_{0}=\overline{\partial u_{0}}$ and $\int u_{0}(z)~dA(z)=0$. Then
\begin{equation}
q(t)=\mathcal{M}\left(  \mathcal{I}\left(  e^{2it\left(  \left(
\diamond\right)  ^{2}+\left(  \overline{\diamond}\right)  ^{2}\right)
}\left(  \mathcal{R}u_{0}\right)  (\diamond)\right)  \right)
\label{eq:NV.sol.bis}%
\end{equation}
is a weak solution the NV\ equation with initial data $q_{0}$. If $u_{0}%
\in\mathcal{S}(\mathbb{R}^{2})$, then $q(t)$ is a classical solution of the NV\ equation.
\end{theorem}

The class of initial data covered by Theorem \ref{thm:NV} includes the
conductivity-type potentials considered by Lassas, Mueller, Siltanen, and
Stahel. The connection between their work and ours is given in the following theorem.

\begin{theorem}
\label{thm:PP.and.LMSS}Suppose that $u_{0}\in\mathcal{C}_{0}^{\infty
}(\mathbb{R}^{2})$ with $\int u_{0}(z)~dA(z)=0$ and $\overline{\partial u_{0}%
}=\partial u_{0}$, and let $q_{0}=2\partial u_{0}+\left\vert u_{0}\right\vert
^{2}$. Then, for any $t$,%
\[
\mathcal{Q}\left(  e^{it\left(  \left(  \diamond\right)  ^{3}+\left(
\overline{\diamond}\right)  ^{3}\right)  }(\mathcal{T}q_{0})(\diamond)\right)
=\mathcal{MI}\left(  e^{t\left(  \left(  \overline{\diamond}\right)
^{3}-(\diamond)^{3}\right)  }\left(  \mathcal{R}u_{0}\right)  (\diamond
)\right)
\]
and their common value is a classical solution to the Novikov-Veselov equation.
\end{theorem}

It should be noted that the solution formula (\ref{eq:NV.sol.bis}) provides a
solution which exists \emph{globally} in time. On the other hand, Taimanov and
Tsarev (see \cite{TT:2007,TT:2008a,TT:2008b,TT:2010}) have used Moutard
transformations to construct explicit, nonsingular Cauchy data $q_{0}$ with
rapid decay at infinity and having the following properties: (i) the
Schr\"{o}dinger operator $-\Delta+q_{0}$ has nonzero eigenvalues at zero
energy and (ii) the solution of (\ref{eq:NV}) with Cauchy data $q_{0}$ blows
up in finite time.

To close this introduction, we comment on the seemingly restrictive hypothesis
in Theorems \ref{thm:NV} and \ref{thm:PP.and.LMSS}. In both theorems, we
assume that $\int u_{0}=0$. To understand what this assumption means, we
recall that if $\phi_{0}=\overline{\partial}^{-1}u_{0}$, then the unique,
positive, normalized zero-energy solution of the Schr\"{o}dinger equation
(\ref{eq:SE1}) is given by $\psi_{0}=\exp\left(  \phi_{0}\right)  $. For
$u_{0}\in\mathcal{S}(\mathbb{R}^{2})$ say, we have from the integral
expression for $\overline{\partial}^{-1}$ that
\[
\phi_{0}(z)=-\frac{1}{\pi}\frac{\int u_{0}(\zeta)~d\zeta}{z}+\mathcal{O}%
\left(  \left\vert z\right\vert ^{-2}\right)
\]
so that, to leading order%
\[
\psi_{0}-1=-\frac{1}{\pi}\frac{\int u_{0}(\zeta)~d\zeta}{z}+\mathcal{O}\left(
\left\vert z\right\vert ^{-2}\right)
\]
Recalling that $\gamma^{1/2}(z)=\psi_{0}(z)$ we see that the vanishing of
$\int u_{0}(z)~dA(z)$ implies that $\gamma(z)-1=\mathcal{O}\left(  \left\vert
z\right\vert ^{-2}\right)  $ as $\left\vert z\right\vert \rightarrow\infty$.
In particular, for conductivities with $\gamma=1$ outside a compact set, $\int
u_{0}(z)~dA(z)=0$.

Indeed, suppose that $q=\gamma^{-1/2}\Delta\left(  \gamma^{1/2}\right)  $ in
distribution sense, where $\gamma\in L^{\infty}(\mathbb{R}^{2})$,
$\gamma(z)\geq c>0$, and suppose further that $\Delta\left(  \nabla
\gamma\right)  $ and $\gamma-1$ belong to $L^{2}(\mathbb{R}^{2})$. It follows
that $\varphi=\log\gamma\in H^{3,1}(\mathbb{R}^{2})$\ and the function%
\[
u=2\overline{\partial}\varphi
\]
belongs to $H^{2,1}$. We then compute that $q=2\partial u+\left\vert
u\right\vert ^{2}$. If we have stronger decay of $\gamma\left(  z\right)  $ as
$\left\vert z\right\vert \rightarrow\infty$, this will imply additional decay
of $\varphi(z)$ that can be used to check $\int u(z)~dA(z)=0$ by Green's
formula $\int_{\Omega}\overline{\partial}\varphi~dA(z)=\frac{1}{2}%
\int_{\partial\Omega}\varphi\left(  \nu_{x_{1}}+i\nu_{x_{2}}\right)  ~d\sigma$.

The structure of this paper is as follows. In \S \ref{sec:prelim} we review
some important linear and multilinear estimates which will be used to study
the scattering maps $\mathcal{R}$ and $\mathcal{I}$. In \S \ref{sec:IR.def} we
recall how the scattering maps $\mathcal{R}$ and $\mathcal{I}$ for the
Davey-Stewartson system are defined, while in \S \ref{sec:IR} we prove that
$\mathcal{R}:H^{2,1}(\mathbb{R}^{2})\rightarrow H^{1,2}(\mathbb{R}^{2})$ and
$\mathcal{I}:H^{1,2}(\mathbb{R}^{2})\rightarrow H^{2,1}(\mathbb{R}^{2})$ are
locally Lipschitz continuous. In \S \ref{sec:mNV} we solve the mNV\ equation
using the inverse scattering method and prove that, for initial data $u_{0}\in
H^{2,1}(\mathbb{R}^{2})$ with $\partial u_{0}=\overline{\partial u_{0}}$ and
$\int_{\mathbb{R}^{2}}u_{0}(z)~dA(z)=0$, the condition $\partial
u=\overline{\partial u}$ holds for all $t>0$. In \S \ref{sec:conductivity} we
show that our class of potentials extends the class of conductivity type
potentials considered by Lassas, Mueller, Siltanen and Stahel \cite{LMSS:2011}%
, and that our solution coincides with theirs where the two constructions
overlap. Appendix \ref{app:mNV} sketches the solution of the mNV equation by
scattering theory for initial data in the Schwarz class.

\emph{Acknowledgements}. The author gratefully acknowledges the support of the
College of Arts and Sciences at the University of Kentucky for a CRAA\ travel
grant and the Isaac Newton Institute for hospitality during part of the time
this work was done. The author thanks Fritz Gesztesy and Russell Brown for
helpful conversations and correspondence.

\section{Preliminaries}

\label{sec:prelim}

\emph{Notation}. In what follows, $\left\Vert ~\cdot~\right\Vert _{p}$ denotes
the usual $L^{p}$-norm and $p^{\prime}=p/(p-1)$ denotes the conjugate
exponent. If $f$ is a function of $\left(  z,k\right)  $, $f(z,\diamond)$
(resp. $f(~\cdot~,k)$) denotes $f$ with a generic argument in the $z$ or $k$
variable. We will write $L_{z}^{p}$ or $L_{k}^{p}$ for $L^{p}$-spaces with
respect to the $z$ or $k$ variable, and $L_{z}^{p}\left(  L_{k}^{q}\right)  $
for the mixed spaces with norm%
\[
\left\Vert f\right\Vert _{L_{z}^{p}\left(  L_{k}^{q}\right)  }=\left(
\int\left\Vert f(z,~\diamond~)\right\Vert _{q}^{p}~dA(z)\right)  ^{1/p}.
\]
If $f$ is a function of $z$ and $k$, $\left\Vert f\right\Vert _{\infty}$
denotes $\left\Vert f\right\Vert _{L^{\infty}\left(  \mathbb{R}_{z}^{2}%
\times\mathbb{R}_{k}^{2}\right)  }$.

In what follows, $\left\langle ~\cdot~,~\cdot~\right\rangle $ denotes the
pairing%
\[
\left\langle f,g\right\rangle =-\frac{1}{\pi}\int\overline{f(z)}g(z)~dA(z)
\]

We will call a mapping $f$ from a Banach space $X$ to a Banach space $Y$ a
\emph{LLCM} (locally Lipschitz continuous map) if for any bounded subset $B$
of $X$, there is a constant $C=C(B)$ so that, for all $x_{1},x_{2}\in B$,
\[
\left\Vert f(x_{1})-f(x_{2})\right\Vert _{Y}\leq C(B)\left\Vert x_{1}%
-x_{2}\right\Vert _{X}.
\]
For example, if $M:X^{m}\rightarrow Y$ is a continuous multilinear map, then
\[
f\mapsto M(f,f,\ldots,f)
\]
is a LLCM from $X$ to $Y$.

\emph{Cauchy Transforms}. The integral operators%
\begin{align*}
P\psi &  =\frac{1}{\pi}\int\frac{1}{\zeta-z}~f(\zeta)~dm(\zeta),\\
\overline{P}\psi &  =\frac{1}{\pi}\int\frac{1}{\overline{\zeta}-\overline{z}%
}~f(\zeta)~dm(\zeta)
\end{align*}
are formal inverses respectively of $\overline{\partial}$ and $\partial$. We
denote by $P_{k}$ and $\overline{P}_{k}$ the corresponding formal inverses of
$\overline{\partial}_{k}$ and $\partial_{k}$. The following estimates are
standard (see, for example, Astala-Iwaniec-Martin \cite{AIM:2009}, \S 4.3, or
Vekua \cite{Vekua:1962}).

\begin{lemma}
\label{lemma:P}(i) For any $p\in(2,\infty)$ and $f\in L^{p}$, $\left\Vert
Pf\right\Vert _{p}\leq C_{p}\left\Vert f\right\Vert _{2p/(p+2)}.$\newline(ii)
For any $p,q$ with $1<q<2<p<\infty$ and any $f\in L^{p}\cap L^{q}$,
$\left\Vert Pf\right\Vert _{\infty}\leq C_{p,q}\left\Vert f\right\Vert
_{L^{p}\cap L^{q}}$ and $Pf$ is H\"{o}lder continuous of order $(p-2)/p$ with
\[
\left\vert \left(  Pf\right)  (z)-\left(  Pf\right)  (w)\right\vert \leq
C_{p}\left\vert z-w\right\vert ^{(p-2)/p}\left\Vert f\right\Vert _{p}.
\]
\newline(iii) For $2<p<q$ and $u\in L^{s}$ for $q^{-1}+1/2=p^{-1}+s^{-1}$ $,$%
\[
\left\Vert P(u\psi)\right\Vert _{q}\leq C_{p,q}\left\Vert u\right\Vert
_{s}\left\Vert \psi\right\Vert _{p}.
\]

\end{lemma}

\begin{remark}
If $p>2$ and $u\in L^{s}$ for $s\in(1,\infty)$, then estimate (iii) holds true
for any $q>2$.
\end{remark}

\emph{Beurling Transform}. The operator
\begin{equation}
\left(  \mathcal{S}f\right)  (z)=-\frac{1}{\pi}\int\frac{1}{\left(
z-w\right)  ^{2}}f(w)~dw\label{eq:AB}%
\end{equation}
defined as a Calderon-Zygmund type singular integral, has the property that
for $f\in C_{0}^{\infty}(\mathbb{R}^{2})$ we have $\mathcal{S}\left(
\overline{\partial}f\right)  =\partial f$. The operator $\mathcal{S}$ is a
bounded operator on $L^{p}$ for $p\in\left(  1,\infty\right)  $ (see for
example \cite{AIM:2009}, \S 4.5.2). This fact allows us to obtain $L^{p}%
$-estimates on $\partial$-derivatives of functions of interest from $L^{p}%
$-estimates on $\overline{\partial}$-derivatives.

\emph{Brascamp-Lieb Type Estimates}. \ A fundamental role is played by the
following multilinear estimate due to Russell Brown \cite{Brown:2001}, who
initiated their use in the analysis of the DS\ II scattering maps. See
Appendix A of \cite{P}, written by Michael Christ, for a proof of these
estimates using the methods of Bennett, Carbery, Christ, and Tao
\cite{BCCT1,BCCT2}. Define%
\[
\Lambda_{n}(\rho,u_{0},u_{1},\ldots,u_{2n})=\int_{\mathbb{C}^{2n+1}}%
\frac{\left\vert \rho(\zeta)\right\vert \left\vert u_{0}(z_{0})\right\vert
\ldots\left\vert u(z_{2n})\right\vert }{\prod_{j=1}^{2k}\left\vert
z_{j-1}-z_{j}\right\vert }dA(z),
\]
where $dA(z)$ is product measure on $\mathbb{C}^{2n+1}$, and set
\begin{equation}
\zeta=\sum_{j=0}^{2n}(-1)^{j}z_{j}. \label{eq:zeta}%
\end{equation}

\begin{proposition}
\cite{Brown:2001}\label{prop:brown} The estimate%
\begin{equation}
\left\vert \Lambda_{n}(\rho,u_{0},u_{1},\ldots,u_{2n})\right\vert \leq
C_{n}\left\Vert \rho\right\Vert _{2}\prod_{j=0}^{2n}\left\Vert u_{j}%
\right\Vert _{2} \label{ineq:brown}%
\end{equation}
holds.
\end{proposition}

\begin{remark}
\label{rem:multi1}For $u_{1},\ldots,u_{2n}\in\mathcal{S}(\mathbb{R}^{2})$,
define operators $W_{j}$ by $W_{j}\psi=Pe_{k}u_{j}\overline{\psi}$.
Proposition \ref{prop:brown} implies that
\[
F(k)=\left\langle e_{k}u_{0},W_{1}W_{2}\ldots W_{2n}1\right\rangle
\]
is a multilinear $L^{2}(\mathbb{R}^{2})$-valued function of $\left(
u_{0},\ldots,u_{2n}\right)  $ with
\[
\left\Vert F\right\Vert _{2}\leq C\prod_{j=0}^{2n}\left\Vert u_{j}\right\Vert
_{2}.
\]

\end{remark}

\section{Scattering Maps and an Oscillatory $\overline{\partial}$-Problem}

\label{sec:IR.def}

First, we recall from \cite{P} that the Davey-Stewartson scattering maps
$\mathcal{R}$ and $\mathcal{I}$ are both defined by $\overline{\partial}%
$-problems: see \cite{P} for full discussion. The inverse scattering method
for the Davey-Stewartson II\ equation was developed Ablowitz-Fokas
\cite{AF:1983,AF:1984} and Beals-Coifman
\cite{BC:1984,BC:1985,BC:1989,BC:1990}. Sung \cite{Sung:1994} and Brown
\cite{Brown:2001} carried out detailed analytical studies of the map.

For a complex parameter $k$ and for $z=x_{1}+ix_{2}$, let%
\[
e_{k}=e^{\overline{k}\overline{z}-kz}%
\]
Given $u\in H^{1,1}(\mathbb{R}^{2})$ and $k\in\mathbb{C}$, there exists a
unique bounded continuous solution of
\begin{align}
\overline{\partial}\mu_{1} &  =\frac{1}{2}e_{k}u\overline{\mu_{2}%
},\label{eq:mu.dbar}\\
\overline{\partial}\mu_{2} &  =\frac{1}{2}e_{k}u\overline{\mu_{1}},\nonumber\\
\lim_{\left\vert z\right\vert \rightarrow\infty}\left(  \mu_{1}(z,k),\mu
_{2}(z,k)\right)   &  =\left(  1,0\right)  .\nonumber
\end{align}
We then define $r=\mathcal{R}u$ by%
\begin{equation}
r(k)=-\frac{1}{\pi}\int e_{k}(z)u(z)\overline{\mu_{1}(z,k)}%
~dA(z).\label{eq:r.rep}%
\end{equation}

On the other hand, it can be shown that
\begin{equation}
(\nu_{1},\nu_{2})=\left(  \mu_{1},e_{k}\overline{\mu_{2}}\right)
\label{eq:nu.def}%
\end{equation}
solve a $\overline{\partial}$-problem in the $k$ variable:
\begin{align}
\overline{\partial}_{k}\nu_{1} &  =\frac{1}{2}e_{k}\overline{r}\overline
{\nu_{2}},\label{eq:nu.dbar}\\
\overline{\partial}_{k}\nu_{2} &  =\frac{1}{2}e_{k}\overline{r}\overline
{\nu_{1}},\nonumber\\
\lim_{\left\vert k\right\vert \rightarrow\infty}\left(  \nu_{1}(z,k),\nu
_{2}(z,k)\right)   &  =\left(  1,0\right)  ,\nonumber
\end{align}
and that this solution is unique within the bounded continuous functions.
Given $r\in H^{1,1}(\mathbb{R}^{2})$, we solve the $\overline{\partial}%
$-system (\ref{eq:nu.dbar}) and define $u=\mathcal{I}r$ by%
\begin{equation}
u(z)=-\frac{1}{\pi}\int e_{-k}(z)r(k)\nu_{1}(z,k)~dA(k).\label{eq:u.rep}%
\end{equation}
In \cite{P}, we proved:

\begin{theorem}
\label{thm:Lip.pre}The maps $\mathcal{R}$ and $\mathcal{I}$, initially defined
on $\mathcal{S}(\mathbb{R}^{2})$, extend to LLCM's from $H^{1,1}%
(\mathbb{R}^{2})$ to itself. Moreover $\mathcal{R}\circ\mathcal{I}%
=\mathcal{I}\circ\mathcal{R}=I$, where $I$ denotes the identity map on
$H^{1,1}(\mathbb{R}^{2})$.
\end{theorem}

We now describe three basic tools used in \cite{P} to analyze the generic
system%
\begin{align}
\overline{\partial}w_{1}  &  =\frac{1}{2}e_{k}u\overline{w_{2}}%
,\label{eq:w.dbar}\\
\partial w_{2}  &  =\frac{1}{2}e_{k}u\overline{w_{1}},\nonumber\\
\lim_{\left\vert z\right\vert \rightarrow\infty}\left(  w_{1}(z,k),w_{2}%
(z,k)\right)   &  =\left(  1,0\right) \nonumber
\end{align}
for unknown functions $w_{1}(z,k)$ and $w_{2}(z,k)$, where $k$ is a complex
parameter, $u\in H^{1,1}(\mathbb{R}^{2})$. We refer the reader to \cite{P} for
the proofs. We don't state the obvious analogues of the facts below when the
roles of $k$ and $z$ are reversed, but use them freely in what follows.

1. \emph{Finite }$L^{p}$-\emph{Expansions}. In \cite{P} it is shown that the
system (\ref{eq:w.dbar}) has a unique solution in $L_{z}^{\infty}$. This
result, and further analysis of the solution, follows from the following facts
that we recall from \S 3 of \cite{P}. Let $T$ be the antilinear operator%
\[
T\psi=\frac{1}{2}Pe_{k}u\overline{\psi}%
\]
which is a bounded operator from $L^{p}$ to itself for $p\in(2,\infty]$ if
$u\in H^{1,1}$ by\ Lemma \ref{lemma:P}(i). The system (\ref{eq:w.dbar}) is
equivalent to the integral equation%
\[
w_{1}=1+T^{2}w_{1}%
\]
and the auxiliary formula $w_{2}=Tw_{1}$. The operator $I-T^{2}$ has trivial
kernel as a map from $L^{p}(\mathbb{R}^{2})$ to itself for any $p\in
(2,\infty]$, and the estimate
\[
\left\Vert T^{2}\right\Vert _{L^{p}\rightarrow L^{p}}\leq C_{p}\left\Vert
u\right\Vert _{H^{1,1}}^{2}\left(  1+\left\vert k\right\vert \right)  ^{-1}%
\]
holds for any $p\in(2,\infty)$.  For any $p\in(2,\infty)$, the resolvent
$\left(  I-T^{2}\right)  ^{-1}$ is bounded uniformly in $k\in\mathbb{C}$ and
$u$ in bounded subsets of $H^{1,1}$ as an operator from $L^{p}$ to itself.
Note that if $u\in H^{1,1}$, the expression $T1=\frac{1}{2}Pe_{k}u$ is a
well-defined element of $L^{p}$ for all $p\in(2,\infty]$. The unique solution
of (\ref{eq:w.dbar}) is given by
\begin{align*}
w_{1}-1 &  =\left(  I-T^{2}\right)  ^{-1}T^{2}1,\\
w_{2} &  =Tw_{1}.
\end{align*}
From these facts, one has (see \S 3 of \cite{P}):

\begin{lemma}
\label{lemma:1}(Finite $L^{p}$-expansions) For any positive integer $N$, the
expansions%
\begin{align*}
w_{1}-1  &  =\sum_{j=1}^{N}T^{2j}1+R_{1,N}\\
w_{2}  &  =\sum_{j=1}^{N}T^{2j-1}1+R_{2,N}%
\end{align*}
hold, where the maps
\begin{align*}
u  &  \mapsto\left(  1+\left\vert \diamond\right\vert \right)  ^{N}%
R_{1,N}(~\cdot~,~\diamond~),\\
u  &  \mapsto\left(  1+\left\vert \diamond\right\vert \right)  ^{N}%
R_{2,N}(~\cdot~,~\diamond~)
\end{align*}
are LLCM's\ from $H^{1,1}\left(  \mathbb{R}^{2}\right)  $ into $L_{k}^{\infty
}\left(  L_{z}^{p}\right)  .$
\end{lemma}

2. \emph{Multilinear Estimates}. Substituting the expansions into the
representation formulas (\ref{eq:u.rep})\ and (\ref{eq:r.rep}) leads to
expressions of the form%
\[
\left\langle e_{\ast}w,F_{j}\right\rangle
\]
where $e_{\ast}$ denotes $e_{k}$ or $e_{-k}$, $w$ is a monomial in $u$ and its
derivatives, and $F_{j}$ denotes $T^{2j}1$ or $\overline{T^{2j}1}$ for
$j\geq1$. We assume that $w$ is bounded in $L^{2}$ norm by a power of
$\left\Vert u\right\Vert _{H^{2,1}}$. The following fact is an immediate
consequence of Remark \ref{rem:multi1}.

\begin{lemma}
\label{lemma:2}The map $u\mapsto\left\langle e_{\ast}w,F_{j}\right\rangle $ is
a LLCM from $H^{2,1}\left(  \mathbb{R}^{2}\right)  $ to $L_{k}^{2}\left(
\mathbb{R}^{2}\right)  $.
\end{lemma}

3. \emph{Large-Parameter Expansions}. Finally, the following large-$z$ finite
expansions for $w_{1}$ and $w_{2}$ will be useful. We omit the straightforward
computational proof.

\begin{lemma}
\label{lemma:3}For $u\in H^{1,1}\left(  \mathbb{R}^{2}\right)  $,
\begin{align*}
w_{1}(z,k)-1  &  =-\frac{1}{2\pi z}\int e_{k}(z^{\prime})u(z^{\prime
})\overline{w_{2}(z^{\prime},k)}~dm(z^{\prime})\\
&  -\frac{1}{2\pi z}\int\frac{e_{k}(z^{\prime})}{z-z^{\prime}}z^{\prime
}u(z^{\prime})\overline{w_{2}(z^{\prime},k)}~dm(z^{\prime})
\end{align*}
and similarly%
\begin{align*}
w_{2}(z,k)  &  =-\frac{1}{2\pi z}\int e_{k}(z^{\prime})u(z^{\prime}%
)\overline{w_{1}(z^{\prime},k)}~dm(z^{\prime})\\
&  -\frac{1}{2\pi z}\int\frac{e_{k}(z^{\prime})}{z-z^{\prime}}z^{\prime
}u(z^{\prime})\overline{w_{1}(z^{\prime},k)}~dm(z^{\prime})
\end{align*}

\end{lemma}

Analogous expansions hold for the $\overline{\partial}$-problem in the $k$ variables.

\section{Restrictions of Scattering Maps}

\label{sec:IR}

In this section we prove Theorem \ref{thm:Lip}. In virtue of Theorem
\ref{thm:Lip.pre}, it suffices to show that the maps $H^{2,1}\ni
u\mapsto\left\vert \diamond\right\vert ^{2}r\left(  \diamond\right)  $ and
$H^{1,2}\ni r\mapsto\Delta u\in L^{2}$ are LLCM's. First, we prove:

\begin{lemma}
\label{lemma:restrict.R}The map $u\mapsto\left\vert \diamond\right\vert
^{2}r\left(  \diamond\right)  $ is a LLCM\ from $H^{2,1}(\mathbb{R}^{2})$ to
$L^{2}(\mathbb{R}^{2})$.
\end{lemma}

\begin{proof}
We carry out all computations on $u$ $\in\mathcal{C}_{0}^{\infty}%
(\mathbb{R}^{2})$ and extend by density to $H^{2,1}(\mathbb{R}^{2})$. Note
that $\left\Vert u\right\Vert _{p}\leq C_{p}\left\Vert u\right\Vert _{H^{2,1}%
}$ for all $p\in\left(  1,\infty\right)  $ and $\left\Vert \partial
u\right\Vert _{p}$ $\leq\left\Vert u\right\Vert _{H^{2,1}}$ for $p\in
\lbrack2,\infty)$. An integration by parts using (\ref{eq:r.rep}) and the
identity $\partial e_{k}=-ke_{k}$ shows that (up to trivial factors)%
\begin{align*}
\left\vert k\right\vert ^{2}r(k)  &  =-\overline{k}\int e_{k}\left(  \partial
u\right)  -\overline{k}\int e_{k}\left(  \partial u\right)  \left(
\overline{\mu_{1}}-1\right)  -\frac{\overline{k}}{2}\int\left\vert
u\right\vert ^{2}\mu_{2}\\
&  =I_{1}+I_{2}+I_{3}%
\end{align*}
where in the last term we used
\begin{equation}
\overline{\partial}\mu_{1}=\frac{1}{2}e_{k}u\overline{\mu_{2}}.
\label{eq:mu1.dbar}%
\end{equation}

$I_{1}$: This term is the Fourier transform of $\partial\overline{\partial}u$
and hence defines a linear map from $H^{2,1}$ to $L_{k}^{2}$.

$I_{2}$: An integration by parts using (\ref{eq:r.rep}), the identity
$\partial\left(  e_{k}\right)  =-ke_{k}$, and (\ref{eq:mu1.dbar}) again shows
that%
\begin{align*}
I_{2} &  =\frac{\overline{k}}{k}\left[  \int e_{k}\left(  \partial
^{2}u\right)  \left(  \overline{\mu_{1}}-1\right)  +\frac{1}{2}\int
\overline{u}\partial u~\mu_{2}\right]  \\
&  =I_{21}+I_{22}.
\end{align*}
In $I_{21}$ we insert $1=\chi+\left(  1-\chi\right)  $ where $\chi
\in\mathcal{C}_{0}^{\infty}(\mathbb{R}^{2})$ satisfies $0\leq\chi(z)\leq1$,
$\chi(z)=1$ for $\left\vert z\right\vert \leq1$, and $\chi(z)=0$ for
$\left\vert z\right\vert \geq2$. Drop the unimodular factor $\overline{k}/k$
and write $I_{21}=I_{21}^{\mathrm{in}}+I_{21}^{\mathrm{out}}$ corresponding to
this decomposition. Since $\chi\partial^{2}u$ $\in L^{p^{\prime}}$ for any
$p>2$, we may expand%
\[
I_{21}^{\mathrm{in}}=\sum_{j=1}^{N}\int e_{k}\left(  \partial^{2}u\right)
\chi\left(  \overline{T^{2j}1}\right)  +\int e_{k}\left(  \partial
^{2}u\right)  \chi\overline{\left(  I-T^{2}\right)  ^{-1}T^{2j+2}1}%
\]
By Lemma \ref{lemma:1}, Lemma \ref{lemma:2} and the fact that $\chi
\partial^{2}u\in L^{p^{\prime}}$, each right-hand term defines a LLCM from
$H^{2,1}$ to $L_{k}^{2}$, hence $u\mapsto I_{21}^{\mathrm{rm}}$ is a LLCM. In
$I_{21}^{\mathrm{out}}$,\ we use Lemma \ref{lemma:3} to write%
\begin{align}
\int e_{k}\left(  1-\chi\right)  \partial^{2}u\left(  \overline{\mu_{1}%
}-1\right)   &  =-\frac{1}{2\pi}\left(  \int e_{k}\left(  1-\chi\right)
\left(  \partial^{2}u\right)  z^{-1}\right)  \left(  \int e_{-k}\overline
{u}\mu_{2}\right)  \label{eq:I21.1}\\
&  +\frac{1}{2}\left\langle e_{-k}\left(  1-\chi\right)  \overline{\left(
\partial^{2}u\right)  z^{-1}},Pe_{-k}u_{1}\left(  T\mu_{1}\right)
\right\rangle .\nonumber
\end{align}
The first right-hand term in (\ref{eq:I21.1}) is the product of the Fourier
transform of the $L^{2}$-function $\left(  1-\chi(z)\right)  (\partial
^{2}u)(z)z^{-1}$ and the function $\int e_{-k}\overline{u}\mu_{2}$. Since $u$
$\in L^{p^{\prime}}$ for all $p>2$ while $u\mapsto\mu_{2}$ is a LLCM\ from
$H^{1,1}$ to $L_{k}^{\infty}\left(  L_{z}^{p}\right)  $, the map $u\mapsto\int
e_{-k}\overline{u}\mu_{2}$ is a LLCM from $H^{2,1}$ to $L_{k}^{\infty}$, so
the first right-hand term in (\ref{eq:I21.1}) defines a LLCM\ from $H^{2,1}$
to $L_{k}^{2}$. The second right-hand term in (\ref{eq:I21.1}) may be
controlled using Lemmas \ref{lemma:1} and \ref{lemma:2}. This shows that
$u\mapsto I_{21}^{\mathrm{out}}$, and hence $u\mapsto I_{21}$, defines a LLCM
from $H^{2,1}$ to $L_{k}^{2}$. Finally, to control $I_{22}$, we note that
$\overline{u}\partial u$ $\in L^{p^{\prime}}$ for $p>2$. Hence, using Lemma
\ref{lemma:1} we obtain%
\begin{equation}
I_{22}=\sum_{j=0}^{N}\int\overline{u}\partial u~T^{2j+1}1+\int\left(
\overline{u}\partial u\right)  ~\left(  I-T^{2}\right)  T^{2j+1}%
1.\label{eq:I22}%
\end{equation}
To control terms in the finite sum in (\ref{eq:I22}), we write%
\begin{align*}
\int\overline{u}\partial u~T^{2j+1}1 &  =\left\langle u\partial\overline
{u},P\left[  e_{k}u\left(  \overline{T^{2j}1}\right)  \right]  \right\rangle
\\
&  =-\left\langle e_{-k}\overline{u}\overline{P}\left(  u\partial\overline
{u}\right)  ,\overline{T^{2j}1}\right\rangle .
\end{align*}
and apply Lemma 2 since $\left\Vert u\overline{P}\left(  u\partial\overline
{u}\right)  \right\Vert _{2}^{^{\prime}}\leq C\left\Vert u\right\Vert
_{H^{2,1}}$. The second right-hand term in (\ref{eq:I22}) defines a LLCM from
$H^{2,1}$ to $L_{k}^{2}$ by Lemma \ref{lemma:1}. Hence, $u\mapsto I_{2}$ is a
LLCM from $H^{2,1}$ to $L_{k}^{2}$.

$I_{3}$: Note that $\left\vert u\right\vert ^{2}\in L^{p^{\prime}}$ for all
$p>2$ and use the expansion of $\mu_{2}$ to write $I_{3}$ as
\[
\sum_{j=1}^{N}-\frac{\overline{k}}{2}\int\left\vert u\right\vert ^{2}%
T^{2j+1}1-\frac{\overline{k}}{2}\int\left\vert u\right\vert ^{2}\left(
I-T^{2}\right)  ^{-1}T^{2N+3}1.
\]
The remainder is a LLCM\ from $H^{2,1}$ to $L_{k}^{2}$ by Lemma \ref{lemma:1}.
A given term in the finite sum is written (up to constant factors)%
\begin{align}
\overline{k}\left\langle \left\vert u\right\vert ^{2},P\left[  e_{k}u\left(
\overline{T^{2j}1}\right)  \right]  \right\rangle  &  =\overline
{k}\left\langle e_{-k}\overline{u}\overline{P}\left(  \left\vert u\right\vert
^{2}\right)  ,\overline{T^{2j}1}\right\rangle \label{eq:nuke.the.kbar}\\
&  =-\left\langle \overline{\partial}\left(  e_{-k}\overline{u}\overline
{P}\left(  \left\vert u\right\vert ^{2}\right)  \right)  ,\overline{T^{2j}%
1}\right\rangle \nonumber\\
&  +\left\langle e_{-k}\overline{\partial}\left(  \overline{u}\overline
{P}\left(  \left\vert u\right\vert ^{2}\right)  \right)  ,\overline{T^{2j}%
1}\right\rangle .\nonumber
\end{align}
where we integrated by parts to remove the factor of $\overline{k}$. The first
right-hand term in the second line of (\ref{eq:nuke.the.kbar}) is
\begin{align*}
\left\langle e_{-k}\overline{u}\overline{P}\left(  \left\vert u\right\vert
^{2}\right)  ,\partial\left(  \overline{T^{2j}1}\right)  \right\rangle  &
=\left\langle e_{-k}\overline{u}\overline{P}\left(  \left\vert u\right\vert
^{2}\right)  ,e_{-k}\overline{u}P\left(  e_{k}u\overline{T^{2j-2}1}\right)
\right\rangle \\
&  =\left\langle e_{-k}\overline{u}P\left(  \left\vert u\right\vert
^{2}P\left(  \left\vert u\right\vert ^{2}\right)  \right)  ,\overline
{T^{2j-2}1}\right\rangle
\end{align*}
which defines a LLCM from $H^{2,1}$ to $L_{k}^{2}$ by Lemma \ref{lemma:2}
since $\overline{u}P\left(  \left\vert u\right\vert ^{2}P\left(  \left\vert
u\right\vert ^{2}\right)  \right)  \in L^{2}$. The second right-hand term is
treated similarly. Hence $u\mapsto I_{3}$ is a LLCM\ from $H^{2,1}$ to
$L_{k}^{2}$.

Collecting these results, we conclude that $u\mapsto\left\vert \diamond
\right\vert ^{2}r\left(  \diamond\right)  $ is a LLCM from $H^{2,1}$ to
$L_{k}^{2}$.
\end{proof}

\begin{lemma}
\label{lemma:restrict.I}The map $r\mapsto\Delta u$ is a LLCM\ from
$H^{2,1}(\mathbb{R}^{2})$ to $L^{2}(\mathbb{R}^{2})$.
\end{lemma}

\begin{proof}
Since $r\in H^{1,2}$ we have $kr(k)\in L^{p}$ for all $p\in(1,2]$, $r\in
L^{p}$ for all $p\in\lbrack1,\infty)$ and $\partial r\in L^{p}$ for all
$p\in\lbrack2,\infty)$. A straightforward computation shows that%
\begin{align*}
\partial\overline{\partial}u &  =\int\left\vert k\right\vert ^{2}e_{-k}%
r+\int\left\vert k\right\vert ^{2}e_{-k}r\left(  \nu_{1}-1\right)
-\int\overline{k}e_{-k}r\partial\nu_{1}\\
&  +\int ke_{-k}r\overline{\partial}\nu_{1}+\int e_{-k}r\partial
\overline{\partial}\nu_{1}\\
&  =I_{1}+I_{2}+I_{3}+I_{4}+I_{5}%
\end{align*}
where all derivatives are taken with respect to $z$. We now show that each of
$I_{1}$--$I_{5}$ defines a locally Lipschitz continuous map from $H^{2,1}\ni
r$ into $L_{z}^{2}$.

$I_{1}$: This term is the Fourier transform of $\partial\overline{\partial}r$
and hence $L^{2}$.

$I_{2}$: Inserting $1=\chi+\left(  1-\chi\right)  $ in $I_{2}$, where $\chi$
is as in the proof of Lemma \ref{lemma:restrict.R} (except that, here, $\chi$
is a function of $k$, not $z$), we have $I_{2}=I_{21}+I_{22}$ where%
\[
I_{21}=\int e_{-k}\left\vert k\right\vert ^{2}\chi r\left(  \nu_{1}-1\right)
,~~~I_{22}=\int e_{-k}\left\vert k\right\vert ^{2}r\left(  1-\chi\right)
\left(  \nu_{1}-1\right)  .
\]
We will show that $I_{21}$ and $I_{22}$ are both LLCM's from $H^{1,2}$ to
$L_{z}^{2}$. Since $\left\vert k\right\vert ^{2}\chi r\in L^{p^{\prime}}$ for
any $p>2$, we can use Lemma \ref{lemma:1} for $\nu_{1}-1$ together with Lemma
\ref{lemma:2} to conclude that $r\mapsto I_{21}$ is a LLCM from $H^{1,2}$ to
$L_{z}^{2}$. For $I_{22}$ we use the one-step large-$k$ expansion of $\nu
_{1}-1$ (Lemma \ref{lemma:3}):%
\begin{align*}
\nu_{1}(z,k)-1 &  =-\frac{1}{2\pi k}\int e_{k^{\prime}}(z)\overline
{r(k^{\prime})}~\overline{\nu_{2}(z,k^{\prime})}~dm(k^{\prime})\\
&  -\frac{1}{2\pi k}\int\frac{e_{k^{\prime}}(z)}{k-k^{\prime}}k^{\prime
}r\overline{(k^{\prime})}~\overline{\nu_{2}(z,k^{\prime})}~dm(k^{\prime}).
\end{align*}
We then have
\[
I_{22}=\int e_{-k}\overline{k}r\left(  1-\chi\right)  \left(  F_{1}%
+F_{2}\right)
\]
where%
\begin{align*}
F_{1}(z) &  =-\frac{1}{2\pi}\int e_{k^{\prime}}\overline{r(k^{\prime}%
)}~\overline{\nu_{2}(z,k^{\prime})}~dm(k^{\prime}),~~\\
F_{2}(z,k) &  =-\frac{1}{2\pi}\int\frac{e_{k^{\prime}}(z)}{k-k^{\prime}%
}k^{\prime}r\overline{(k^{\prime})}~\overline{\nu_{2}(z,k^{\prime}%
)}~dm(k^{\prime}).
\end{align*}
It is easy to see that $\left\Vert F_{1}\right\Vert _{L_{z}^{\infty}}%
\leq\left\Vert r\right\Vert _{1}\left\Vert \nu_{2}\right\Vert _{\infty}$, so
that $r\mapsto F_{1}$ is a LLCM\ from $H^{1,2}$ to $L_{z}^{\infty}$. Moreover,
$\int e_{-k}\overline{k}r\left(  1-\chi\right)  $ is the inverse Fourier
transform of the $L^{2}$ function $\overline{\left(  \diamond\right)
}r(\diamond)\left(  1-\chi(\diamond)\right)  $. Hence, $r\mapsto\int
e_{-k}\overline{k}r\left(  1-\chi\right)  F_{1}$ is a LLCM from $H^{1,2}$ to
$L_{z}^{2}$. Next, we may use Lemma \ref{lemma:1} in $F_{2}$ to conclude that%
\begin{equation}
F_{2}=-\frac{1}{2}\sum_{j=1}^{N}P_{k}\left(  e_{k}~k\overline{r}%
~\overline{T^{2j+1}1}\right)  -\frac{1}{2}P_{k}\left(  e_{k}~k\overline
{r}~\overline{\left(  I-T^{2}\right)  ^{-1}T^{2N+3}1}\right)  ,\label{eq:F2}%
\end{equation}
The corresponding contributions to $I_{22}$ from terms in the finite sum from
(\ref{eq:F2}) define LLCM's from $H^{1,2}$ to $L_{z}^{2}$ by Lemma
\ref{lemma:2}, while by the remainder estimate in Lemma \ref{lemma:1}, the
mapping $r\mapsto$ $Pe_{k}~k\overline{r}~\left(  I-T^{2}\right)  ^{-1}%
T^{2N+3}1$ is a LLCM\ from $H^{1,2}$ to $L_{z}^{2}\left(  L_{k}^{p}\right)  $
for $p>2$. Using these estimates we may conclude that $r\mapsto\int
e_{-k}\overline{k}r\left(  1-\chi\right)  F_{2}$ is a LLCM\ from $H^{1,2}$ to
$L_{z}^{2}$.

$I_{3}$: Since $\mu_{1}=\nu_{1}$, we conclude from (\ref{eq:mu1.dbar}) and
(\ref{eq:nu.def}) that
\begin{equation}
\overline{\partial}_{z}\nu_{1}=\frac{1}{2}e_{k}u\overline{\mu_{2}}=\frac{1}%
{2}u\nu_{2} \label{eq:nu1.dbarz}%
\end{equation}
so that%
\begin{align*}
I_{3}  &  =-\int\overline{k}e_{-k}r\left(  \partial\overline{\partial}%
^{-1}\right)  \left(  \overline{\partial}\nu_{1}\right) \\
&  =-\frac{1}{2}\int\overline{k}e_{-k}r\left(  \partial\overline{\partial
}^{-1}\right)  \left(  u\nu_{2}\right)  .
\end{align*}
Porceeding as in the analysis of $I_{22}$ in Lemma \ref{lemma:restrict.R}, we
use the one-step large-$k$ expansion (Lemma \ref{lemma:3}) to obtain%
\begin{align*}
\nu_{2}(z,k)  &  =-\frac{1}{2\pi k}\int e_{k^{\prime}}(z)\overline
{r(k^{\prime})}~\overline{\nu_{2}(z,k^{\prime})}~dm(k^{\prime})\\
&  -\frac{1}{2\pi k}\int\frac{e_{k^{\prime}}(z)}{k-k^{\prime}}k^{\prime
}r\overline{(k^{\prime})}~\overline{\nu_{2}(z,k^{\prime})}~dm(k^{\prime})\\
&  =F_{1}+F_{2}.
\end{align*}
Hence, up to trivial factors,%
\[
I_{3}=\int e_{-k}r\left(  \partial\overline{\partial}^{-1}\right)  \left[
u\left(  F_{1}+F_{2}\right)  \right]  .
\]
By Minkowski's inequality,%
\[
\left\Vert I_{3}\right\Vert _{L_{z}^{2}}\leq\frac{1}{2}\int\left\vert
r\right\vert \left\Vert \partial\overline{\partial}^{-1}\left(  u\left(
F_{1}+F_{2}\right)  \right)  \right\Vert _{L_{z}^{2}}%
\]
Observe that $\left\Vert \partial\overline{\partial}^{-1}\left(
uF_{1}\right)  \right\Vert _{L_{z}^{2}}\leq C\left\Vert uF_{1}\right\Vert
_{L_{z}^{2}}$ while
\begin{align*}
\left\Vert \partial\overline{\partial}^{-1}\left(  uF_{2}\right)  \right\Vert
_{L_{k}^{p}\left(  L_{z}^{2}\right)  }  &  \leq C_{p}\left\Vert u\right\Vert
_{2}\left\Vert F_{2}\right\Vert _{L_{k}^{p}\left(  L_{z}^{\infty}\right)  }\\
&  \leq C_{p}\left\Vert u\right\Vert _{2}\left\Vert \left(  \diamond\right)
r(\diamond)\right\Vert _{2p/(p+2)}\left\Vert \nu_{2}\right\Vert _{\infty}%
\end{align*}
(where $\left\Vert \nu_{2}\right\Vert _{\infty}$ means $\left\Vert
v_{2}\right\Vert _{L^{\infty}(\mathbb{R}_{z}^{2}\times\mathbb{R}_{k}^{2})})$,
so that altogether
\[
\left\Vert I_{3}\right\Vert _{L_{z}^{2}}\leq C\left\Vert u\right\Vert
_{2}\left\Vert r\right\Vert _{H^{1,2}}\left(  1+\left\Vert \nu_{2}\right\Vert
_{\infty}\right)  .
\]
Thus $I_{3}\in L_{z}^{2}$. Local Lipschitz continuity of $I_{3}$ follows from
the local Lipschitz continuity of $r\mapsto u$ and $r\mapsto\nu_{2}$.

$I_{4}$: Using (\ref{eq:nu1.dbarz}) again we compute%
\[
\int ke_{-k}r\overline{\partial}\nu_{1}=\frac{u}{2}\int e_{-k}kr\nu_{2}%
\]
so it suffices to show that $r\mapsto\int e_{-k}kr\nu_{2}$ is a LLCM from
$H^{1,2}$ to $L_{z}^{\infty}$. Since $kr\in L^{p^{\prime}}$ for $p>2$, and
$r\mapsto\nu_{2}$ is a LLCM from $H^{1,1}$ to $L^{\infty}$, the result follows.

$I_{5}$: Compute%
\begin{equation}
I_{5}=\int e_{-k}r\partial\left(  u\nu_{2}\right)  =\partial u\int e_{-k}%
r\nu_{2}+u\int e_{-k}r\left(  \partial\nu_{2}\right)  . \label{eq:I5}%
\end{equation}
The first right-hand term in (\ref{eq:I5}) defines a LLCM from $H^{1,2}$ to
$L_{z}^{2}$ since $r\mapsto\partial u$ has this property. Thus, to control the
right hand term, it suffices to show that $r\mapsto\int e_{-k}r\nu_{2}$
defines a LLCM from $H^{1,2}$ to $L_{z}^{\infty}$. To see this, note that
$r\in L^{p^{\prime}}$ for $p>2$, and $r\mapsto\nu_{2}$ is a LLCM from
$H^{1,1}$ to $L_{z}^{\infty}\left(  L_{z}^{p}\right)  $. To control the second
right-hand term in (\ref{eq:I5}), recall that $\nu_{2}=e_{k}\overline{\mu_{2}%
}$ so that the second term is written%
\begin{equation}
-u\int kre_{k}\overline{\nu_{2}}+\frac{\left\vert u\right\vert ^{2}}{2}\int
e_{-k}r\nu_{1}. \label{eq:I5.2}%
\end{equation}
Since $u$ and $\left\vert u\right\vert ^{2}$ belong to $L^{2}$ it is enough to
show that the two integrals in (\ref{eq:I5.2}) define LLCM's from $r\in
H^{2,1}$ to $L_{z}^{\infty}$. Since $kr\in L^{p^{\prime}}$ for $p>2$ and
$\nu_{2}$ is a LLCM\ from $H^{1,2}$ to $L_{z}^{\infty}\left(  L_{k}%
^{p}\right)  $, the first term in (\ref{eq:I5.2}) clearly has this property.
\ Since $r\in L^{1}$ and $\nu_{1}$ is a LLCM from $r\in H^{2,1}$ to
$L_{z}^{\infty}(L_{k}^{\infty})$, we conclude that the second term also has
this property.
\end{proof}

\section{Solving the mNV\ Equation}

\label{sec:mNV}

In this section we prove Theorem \ref{thm:mNV}. \ Recall that the modified
Novikov-Veselov (mNV) equation \cite{Bogdanov:1987} is:%
\begin{equation}
u_{t}+\left(  \partial^{3}+\overline{\partial}^{3}\right)
u+NL(u)=0\label{eq:mNV}%
\end{equation}
where%
\begin{align*}
NL(u) &  =\frac{3}{4}\left(  \partial\overline{u}\right)  \cdot\left(
\overline{\partial}\partial^{-1}\left(  \left\vert u\right\vert ^{2}\right)
\right)  +\frac{3}{4}\left(  \overline{\partial}u\right)  \cdot\left(
\overline{\partial}\partial^{-1}\left(  \left\vert u\right\vert ^{2}\right)
\right)  \\
&  +\frac{3}{4}\overline{u}\overline{\partial}\partial^{-1}\left(
\overline{u}\overline{\partial}u\right)  +\frac{3}{4}u\partial^{-1}\left(
\overline{\partial}\left(  \overline{u}\overline{\partial}u\right)  \right)  .
\end{align*}
By Theorem \ref{thm:mNV.Schwartz}, for $u_{0}\in\mathcal{S}(\mathbb{R}^{2})$,
the formula%
\begin{equation}
u(z,t)=\mathcal{I}\left(  \exp\left(  \left(  \overline{\diamond}^{3}%
-\diamond^{3}\right)  t\right)  \mathcal{R}u_{0}(\diamond)\right)
(z)\label{eq:mNV.sol.bis}%
\end{equation}
gives a classical solution of the mNV equation. By Lipschitz continuity of
$u_{0}\rightarrow u_{0}(t)$, this formula extends to $u_{0}\in H^{2,1}$, and
exhibits the solution as a continuous curve in $H^{2,1}$ that depends
continuously on the initial data. Since any $u$ given by (\ref{eq:mNV.sol.bis}%
) and $u_{0}\in\mathcal{S}(\mathbb{R}^{2})$ is a classical solution, such a
$u$ trivially satisfies (\ref{eq:mNV.weak}). The same fact for $u(t)$ with
$u_{0}\in H^{2,1}$ follows from the density of $\mathcal{S}(\mathbb{R}^{2})$
in $H^{2,1}$, the continuity of the map (\ref{eq:mNV.sol.bis}) in $u_{0}$, and
an easy approximation argument.

It remains to show that if $u_{0}\in H^{2,1}(\mathbb{R}^{2})\cap
L^{1}(\mathbb{R}^{2})$ and if, also,
\begin{equation}
\int u_{0}dA(z)=0,~~~~\partial u_{0}=\overline{\partial u_{0}}%
,\label{eq:mNV.IC}%
\end{equation}
then $\partial u=\overline{\partial u}$ for all $t$. We will show that this
holds for initial data $u_{0}\in S(\mathbb{R}^{2})$ with the stated
properties, and use Lipschitz continuity of the map $u_{0}\rightarrow u(t)$
defined by (\ref{eq:mNV.sol.bis}) to extend to all $u_{0}\in H^{2,1}%
(\mathbb{R}^{2})\cap L^{1}(\mathbb{R}^{2})$ so that the conditions
(\ref{eq:mNV.IC}) hold.

It will be useful to consider the function%
\[
\varphi=\overline{\partial}^{-1}u
\]
which solves the Cauchy problem%
\begin{align}
\varphi_{t}  &  =-\partial^{3}\varphi-\overline{\partial}^{3}\varphi
\label{eq:mNV.pot}\\
&  -\frac{1}{4}\left(  \partial\varphi\right)  ^{3}-\frac{1}{4}\left(
\overline{\partial}\varphi\right)  ^{3}\nonumber\\
&  +\frac{3}{4}\partial\varphi\cdot\overline{\partial}^{-1}\partial\left(
\left\vert \partial\varphi\right\vert ^{2}\right)  +\frac{3}{4}\overline
{\partial}\varphi\cdot\overline{\partial}^{-1}\partial\left(  \left\vert
\partial\varphi\right\vert ^{2}\right) \nonumber\\
\left.  \varphi\right\vert _{t=0}  &  =\varphi_{0}\nonumber
\end{align}
Note that the condition $\partial u_{0}=\overline{\partial u_{0}}$ implies
that $\varphi_{0}$ is real. On the other hand, to show that $\partial
u=\overline{\partial u}$, it suffices to show that $\varphi$ is real for
$t>0$. To this end, we consider the function%
\[
w=\varphi-\overline{\varphi}%
\]
and derive a linear Cauchy problem satisfied by $w$. \ We will need to know
that $w$ is $L^{2}$ in the space variables.

\begin{lemma}
\label{lemma:mNV.asy}Suppose that $u_{0}\in\mathcal{S}(\mathbb{R}^{2})$, that
$u(t)$ solves the mNV\ equation, and $\varphi(z,t)=\left(  \overline{\partial
}^{-1}u\right)  (t)$. Then for each $t$,
\[
\varphi(z,t)=\frac{c_{0}}{z}+\mathcal{O}\left(  \left\vert z\right\vert
^{-2}\right)
\]
where $c_{0}=\int u_{0}(z)~dA(z)$. If $c_{0}=0$, then $\varphi\in
L^{2}(\mathbb{R}^{2})$ for $t>0$.
\end{lemma}

\begin{proof}
To see that $\varphi$ has the stated form if $u_{0}\in\mathcal{S}%
(\mathbb{R}^{2})$, we note that $u(t)\in\mathcal{S}(\mathbb{R}^{2})$ by the
mapping properties of the scattering transform so that%
\[
\varphi(z,t)=-\frac{1}{\pi z}\int u(z,t)~dt+\mathcal{O}_{t}\left(  \left\vert
z\right\vert ^{-2}\right)
\]
differentiably in $z,t$. Let $c_{0}(t)=\int u(z,t)~dA(z)$. Substituting in
(\ref{eq:mNV.pot}) we easily conclude that $c_{0}^{\prime}(t)=0$. It now
follows that $\varphi(\diamond,t)\in L^{2}(\mathbb{R}^{2})$ as claimed.
\end{proof}

Next, we derive a linear Cauchy problem obeyed by $w$ and derive weighted
estimates on $w$ to show that, if $\left.  w\right\vert _{t=0}=0$, then
$w(t)=0$ identically. It follows that $\varphi$ is real, and hence $\partial
u=\overline{\partial u}$ for all $t>0$. Using (\ref{eq:mNV.pot}) and its
complex conjugate, we easily see that%
\begin{equation}
w_{t}=-\partial^{3}w-\overline{\partial}^{3}w+A\partial w+\overline
{A}\overline{\partial}w \label{eq:w.ev}%
\end{equation}
where%
\[
A=\frac{1}{4}\left[  \left(  \partial\varphi\right)  ^{2}+\left(
\partial\varphi\right)  \cdot\left(  \partial\overline{\varphi}\right)
+\left(  \partial\overline{\varphi}\right)  ^{2}\right]  +\frac{3}{4}%
\overline{\partial}^{-1}\partial\left(  \left\vert \partial\varphi\right\vert
^{2}\right)
\]
Note that $\partial\varphi$, $\overline{\partial}\varphi$ belong to $L^{p}$
for all $p\in(1,\infty)$, uniformly locally in $t$, and that $A$ is smooth
provided that $u_{0}\in\mathcal{S}(\mathbb{R}^{2})$. We will prove:

\begin{lemma}
\label{lemma:weight}Suppose that $A(z,t)$ is a bounded smooth function on
$\mathbb{R}^{2}\times\left(  0,T\right)  $ and that $\eta(z,t)$ is a bounded
smooth nonnegative function with $\left\vert A(z,t)\right\vert \leq\eta(z,t)$
for $z\in\mathbb{C}$ and $t\in\left[  0,T\right]  $. Let $w$ be a smooth
solution of (\ref{eq:w.ev}) with $w(\diamond,t)\in L^{2}(\mathbb{R}^{2})$ for
each $t>0$. Then, there is a constant $C$ so that%
\[
\sup_{t\in\left[  0,T\right]  }\left\Vert w(t)\right\Vert \leq e^{CT}%
\left\Vert w(0)\right\Vert .
\]

\end{lemma}

\begin{proof}
We apply the multiplier method of Chihara \cite{Chihara:2004} (applied to
third-order dispersive nonlinear equations; see Doi \cite{Doi:1994} for a
similar pseudodifferential multiplier method applied to Schr\"{o}dinger-type
equations) to (\ref{eq:w.ev}). Let $\eta$ be a function with%
\[
2\left\vert A(z,t)\right\vert \leq\eta(z,t),
\]
and set%
\[
p_{0}(\xi)=\frac{1}{8}\left(  \xi_{1}^{3}-6\xi_{1}\xi_{2}^{2}\right)  ,
\]
the symbol of the operator $-\partial^{3}-\overline{\partial}^{3}$. With
$z=x_{1}+ix_{2}$ and $\lambda>0$ to be chosen, let
\begin{align}
\gamma(t,x,\xi)  &  =\left(  \int_{-\infty}^{x_{1}}\eta(y,x_{2},t)~dy+\int
_{-\infty}^{x_{2}}\eta(x_{1},y,t)~dy\right) \label{eq:gamma.def}\\
&  \times\frac{\partial p_{0}(\xi)}{\partial\xi_{1}}\frac{\left\vert
\xi\right\vert }{\left\vert \nabla p_{0}(\xi)\right\vert ^{2}}\chi\left(
\frac{\left\vert \xi\right\vert }{\lambda}\right) \nonumber
\end{align}
where $\chi\in\mathcal{C}_{0}^{\infty}([0,\infty))$ is a nonnegative function
with $\chi(t)=0$ for $0\leq t<1/2$ and $\chi(t)=1$ for $t\geq1$. The function
$\gamma$ is constructed so that the principal symbol of the commutator
$\left[  \gamma,p_{0}(D)\right]  $ obeys
\begin{align}
\sigma\left(  \left[  \gamma,p_{0}(D)\right]  \right)   &  =\nabla_{x}%
\gamma(x,\xi,t)\cdot\nabla_{\xi}p_{0}(\xi)\label{eq:gamma.con}\\
&  =\eta(x_{1},x_{2},t)\cdot\left\vert \xi\right\vert \chi\left(
\frac{\left\vert \xi\right\vert }{\lambda}\right)  .\nonumber
\end{align}
By the usual quantization, the pseudodifferential operator $\gamma(t,x,D)$
belongs to the class $OPS^{0}(\mathbb{R}^{n})$. It is easy to see that, also,
the symbols%
\begin{align*}
k(t,x,\xi)  &  =e^{\gamma(t,x,\xi)},\\
\widetilde{k}(t,x,\xi)  &  =e^{-\gamma(t,x,\xi)}%
\end{align*}
define pseudodifferential operators $K(t):=k(t,x,D)$ and $\widetilde
{K}(t):=\widetilde{k}(t,x,D)$ in $OPS^{0}(\mathbb{R}^{n})$ with
\[
K(t)\widetilde{K}(t)-I\in OPS^{-1}(\mathbb{R}^{n})
\]
and%
\[
\lim_{\lambda\rightarrow\infty}\sup_{t\in\lbrack0,T]}\left\Vert K(t)\widetilde
{K}(t)-I\right\Vert =0.
\]
Thus, there is a $\lambda_{0}>0$ so that $K(t)$ is invertible for all
$\left\vert \lambda\right\vert \geq\lambda_{0}$. We take $\left\vert
\lambda\right\vert \geq\lambda_{0}$ from now on.

We claim that if $w(t)$ is a solution of the evolution equation (\ref{eq:w.ev}%
) belonging to $L^{2}(\mathbb{R}^{2})$, the inequality
\begin{equation}
\left\Vert K(t)w(t)\right\Vert \leq\left\Vert K(0)w(0)\right\Vert e^{CT}
\label{eq:w.est}%
\end{equation}
holds for $t\in\left[  0,T\right]  $ and a constant $C$. Since $K(0)$ is
invertible for $\lambda$ sufficiently large, this implies that $w(0)=0$.{}

To prove the inequality (\ref{eq:w.est}), we compute%
\begin{align*}
\frac{d}{dt}\left\Vert K(t)w(t)\right\Vert ^{2}  &  =2\operatorname{Re}\left(
K(t)w(t),\left[  K^{\prime}(t)K^{-1}(t)\right]  K(t)w(t)\right) \\
&  +2\operatorname{Re}\left(  K(t)w(t),K(t)L(t)w(t)\right)
\end{align*}
where%
\[
L(t)=-\partial^{3}-\overline{\partial}^{3}+A\partial+\overline{A}%
\overline{\partial}.
\]
We will show that
\[
K(t)L(t)K(t)^{-1}=-Q_{1}(t)+Q_{2}(t)
\]
where $Q_{1}\left(  t\right)  \in OPS^{1,0}(\mathbb{R}^{2})$ which
$q_{1}(x,\xi):=\sigma(Q_{1}(t))$ nonnegative for $\left\vert \xi\right\vert
\geq2\lambda$, and $Q_{2}(t)\in OPS^{0}(\mathbb{R}^{2})$. If so then by the
sharp G\aa rding inequality \cite{Lax:1966},
\begin{equation}
\operatorname{Re}(v,Q_{1}(t)v)\geq-C_{1}\left\Vert v\right\Vert ^{2}
\label{eq:Q1.lower}%
\end{equation}
and hence
\[
\frac{d}{dt}\left\Vert K(t)w(t)\right\Vert ^{2}\leq C_{3}\left\Vert
K(t)w(t)\right\Vert ^{2}%
\]
where $C_{3}$ majorizes $\left\Vert Q_{2}(t)\right\Vert +\left\Vert K^{\prime
}(t)K^{-1}(t)\right\Vert $. The desired result follows from Gronwall's inequality.

Thus, to finish the proof of (\ref{eq:w.est}), we need only prove that
(\ref{eq:Q1.lower}) holds. But%
\[
K(t)L(t)K(t)^{-1}=K(t)(-p_{0}(D)+A\partial+\overline{A}\overline{\partial
})K(t)^{-1}%
\]
The right-hand side has leading symbol $-q_{1}(x_{1},x_{2},\xi,t)$ where
\[
q_{1}(x_{1},x_{2},\xi,t)=\nabla_{\xi}p_{0}(\xi)\cdot\nabla_{x}\gamma
(t,x_{1},x_{2},\xi)+\operatorname{Re}\left[  A(x_{1},x_{2},t)(\xi_{1}-i\xi
_{2})\right]
\]
which is strictly positive for $\left\vert \xi\right\vert \geq2\lambda$ by
(\ref{eq:gamma.con}). This completes the proof.
\end{proof}

Now suppose that $u_{0}\in\mathcal{S}(\mathbb{R}^{2})$, $\partial
u_{0}=\overline{\partial u_{0}}$, and $\int u_{0}(z)~dA(z)=0$. The function
$\varphi_{0}=\overline{\partial}^{-1}u_{0}$ is real-valued and if $u(t)$
solves the mNV\ equation with Cauchy data $u_{0}$, the function $\varphi
(t)=\left(  \overline{\partial}^{-1}u\right)  (t)$ belongs to $L^{2}%
(\mathbb{R}^{2})$ for all $t$. The same is true of $w(t)=\varphi
(t)-\overline{\varphi(t)}$, and $w(0)=0$. It now follows from Lemma
\ref{lemma:weight} that $w(t)=0$ and $\varphi(t)$ is real-valued for all $t$.
This implies that $\partial u=\overline{\partial u}$ for all $t$.

\begin{proof}
[Proof of Theorem \ref{thm:mNV}]An immediate consequence of Lemma
\ref{lemma:mNV.asy}, Lemma \ref{lemma:weight}, and the above remarks.
\end{proof}

\section{Solving the NV\ Equation}

\label{sec:NV}

In this section we prove Theorem \ref{thm:NV}. The key observation is due to
Bogdanov \cite{Bogdanov:1987} and can be checked by straightforward
computation . Recall the Miura map $\mathcal{M}$, defined in (\ref{eq:Miura2}).

\begin{lemma}
Suppose that $u(z,t)$ is a smooth classical solution of (\ref{eq:mNV}) with
\[
\left(  \partial_{z}u\right)  (z,t)=\overline{\left(  \partial_{z}u\right)
(z,t)},
\]
and $\int u(z,t)~dA(z)=0$ for all $t$. Then, the function
\[
q(z,t)=\mathcal{M}\left(  u(~\cdot~,t)\right)  (z)
\]
is a smooth classical solution of (\ref{eq:NV}).
\end{lemma}

\begin{remark}
In Bogdanov \cite{Bogdanov:1987}, the mNV\ and NV\ are shown to be
gauge-equivalent, and the Miura map is computed from the gauge equivalence.
\end{remark}

\begin{proof}
[Proof of Theorem \ref{thm:NV}]Pick $u_{0}\in H^{2,1}(\mathbb{R}^{2})\cap
L^{1}(\mathbb{R}^{2})$ so that the conditions $\partial u_{0}=\overline
{\partial u_{0}}$ and $\int u_{0}(z)~dA(z)=0$ hold. Let $\left\{
u_{0,n}\right\}  $ be a sequence from $\mathcal{S}(\mathbb{R}^{2})$ with
$u_{n,0}\rightarrow u_{0}$ in $H^{2,1}(\mathbb{R}^{2})\cap L^{1}%
(\mathbb{R}^{2})$. By local Lipschitz continuity of the scattering maps, for
any $T>0$, the sequence $\left\{  u_{n}\right\}  $ from $C([0,T];H^{2,1}%
(\mathbb{R}^{2}))$ given by
\[
u_{n}(z,t)=\mathcal{I}\left(  e^{t\left(  \left(  \diamond\right)
^{3}-\left(  \overline{\diamond}\right)  ^{3}\right)  }\left(  \mathcal{R}%
u_{0,n}\right)  (\diamond)\right)  (z)
\]
converges in $C([0,T];H^{2,1}(\mathbb{R}^{2}))$ to
\[
u(z,t):=\mathcal{I}\left(  e^{t\left(  \left(  \diamond\right)  ^{3}-\left(
\overline{\diamond}\right)  ^{3}\right)  }\left(  \mathcal{R}u_{0}\right)
(\diamond)\right)  (z).
\]
This convergence implies that $q_{n}(z,t):=\mathcal{M}(u_{n}(\diamond,t))(z)$
converges in $L^{2}(\mathbb{R}^{2})$.

Recall (\ref{eq:NV.weak}). Since $q_{n}\rightarrow q$ in $C\left(
[0,T];L^{2}(\mathbb{R}^{2})\right)  $ it follows from the $L^{2}$-bounded of
$\mathcal{S}=\partial\overline{\partial}^{-1}$ that $q_{n}\overline{\partial
}^{-1}\partial q_{n}\rightarrow q\overline{\partial}^{-1}\partial q$ and
$q_{n}\partial^{-1}\overline{\partial}q_{n}\rightarrow q\partial^{-1}%
\overline{\partial}q$ in $C\left(  [0,T],L^{1}(\mathbb{R}^{2}\right)  )$. We
conclude that $q$ is a weak solution of the NV equation.
\end{proof}

\section{Conductivity-Type Potentials}

\label{sec:conductivity}

In this section we show that our solution of NV\ coincides with that of
Lassas, Mueller, Siltanen, and Stahel \cite{LMSS:2011} in the cases they
consider, proving Theorem \ref{thm:PP.and.LMSS}.

We briefly recall some of the notation and results of Lassas, Mueller, and
Siltanen \cite{LMS:2007}. Assume first that $q\in\mathcal{C}_{0}^{\infty
}(\mathbb{R}^{2})$ and is of conductivity type. We denote by $\psi(x,\zeta)$
the unique solution of the problem%
\begin{align}
\left(  -\Delta+q\right)  \psi &  =0,\label{eq:Schr0}\\
\lim_{\left\vert z\right\vert \rightarrow\infty}\left(  e^{-i\left(
x\cdot\zeta\right)  }\psi(x,\zeta)-1\right)   &  =0.\nonumber
\end{align}
where $x=\left(  x_{1},x_{2}\right)  $ and $\zeta\in\mathbb{C}^{2}$ satisfies
$\zeta\cdot\zeta=0$. Here $a\cdot b$ denotes the Euclidean inner product
without complex conjugation. Henceforth, we set $\zeta=\left(  k,ik\right)  $
for $k\in\mathbb{C}$, which amounts to choosing a branch of the variety
$\mathcal{V}=\left\{  \zeta\in\mathbb{C}^{2}:\zeta\cdot\zeta=0\right\}  $.
Since $q$ is of conductivity type, it follows from Theorem 3 in
\cite{Nachman:1995} that the problem (\ref{eq:Schr0}) admits a unique solution
for each $k\in\mathbb{C}$. We set $z=x_{1}+ix_{2}$ and define%
\begin{equation}
m(z,k)=e^{-ikz}\psi(x,\zeta) \label{eq:m}%
\end{equation}
for $\zeta=\left(  k,ik\right)  $.

The direct scattering map%
\begin{equation}
\mathcal{T}:q\rightarrow\mathbf{t}\label{eq:T}%
\end{equation}
is defined by
\begin{equation}
\mathbf{t}(k)=\int e^{i(\overline{k}\overline{z}+kz)}%
q(z)m(z,k)~dA(z)\label{eq:t.rep.bis*}%
\end{equation}
On the other hand, the inverse map
\begin{equation}
\mathcal{Q}:\mathbf{t}\rightarrow q\label{eq:Q}%
\end{equation}
is defined by%
\begin{equation}
q(z)=\frac{i}{\pi^{2}}\overline{\partial}_{z}\left(  \int_{\mathbb{C}}%
\frac{\mathbf{t}(k)}{\overline{k}}e^{-i\left(  kz+\overline{k}\overline
{z}\right)  }\overline{m(z,k)}~dA(k)\right)  \label{eq:q.rep.bis*}%
\end{equation}
where $m(z,k)$ is reconstructed from $t$ via the $\overline{\partial}$-problem%
\begin{equation}
\overline{\partial}_{k}m(x,k)=\frac{\mathbf{t}(k)}{4\pi k}e^{-i\left(
kz+\overline{k}\overline{z}\right)  }(z)\overline{m(x,k)}.\label{eq:m.dbar}%
\end{equation}
Let
\[
\mathbf{m}_{t}^{n}(k)=\exp\left(  -i^{n}\left(  k^{n}+\overline{k}^{n}\right)
t\right)
\]
for an odd positive integer $n$. In \cite{LMSS:2011}, Lassas, Mueller and
Siltanen prove:

\begin{theorem}
\cite{LMSS:2011} For $q_{0}\in\mathcal{C}_{0}^{\infty}(\mathbb{R}^{2})$ radial
and of conductivity type, $\mathcal{QT}(q_{0})=q_{0}$. Moreover, if
\begin{equation}
q(t):=\mathcal{Q}\left(  \mathbf{m}_{t}^{n}\mathcal{T}q_{0}\right)
\label{eq:q.flow}%
\end{equation}
then \thinspace$q(t)$ is a continuous, real-valued potential with $q(t)\in
L^{p}(\mathbb{R}^{2})$ for $p\in(1,2)$.
\end{theorem}

They conjecture that for $n=3$, $q(t)$ given by (\ref{eq:q.flow}) solves the
NV\ equation, provided that $q_{0}$ obeys the hypotheses of Theorem
\cite{LMSS:2011}. We will prove that this is the case (for a larger class of
$q_{0}$) by proving Theorem \ref{thm:PP.and.LMSS}.

We will prove Theorem \ref{thm:PP.and.LMSS} in two steps. First, we show that
for $u\in\mathcal{S}(\mathbb{R}^{2})$ with $\partial u=\overline{\partial u}$
and $\int u_{0}(z)~dA(z)=0$, the scattering data $r=\mathcal{R}u$ is related
to the scattering transform $\mathbf{t}=\mathcal{T}q$ for $q=2\partial
u+\left\vert u\right\vert ^{2}$ by the identity%
\[
\mathbf{t}(k)=-2\pi i\overline{k}~\overline{r(ik)}.
\]
Next, we show that for $\mathbf{t}$ of the above form with $r=\mathcal{R}u$,
the identity%
\[
\left(  \mathcal{Q}\mathbf{t}\right)  (z)=2(\partial u)(z)+\left\vert
u(z)\right\vert ^{2}.
\]
Theorem \ref{thm:PP.and.LMSS} is an easy consequence of these two identities.

The key to both computations is the following construction of complex
geometric optics solutions for the potential $q=\partial u+\left\vert
u\right\vert ^{2}$ from the solutions $\mu=\left(  \mu_{1},\mu_{2}\right)
^{T}$ of (\ref{eq:mu.dbar}). First, suppose that $\Phi=\left(  \Phi_{1}%
,\Phi_{2}\right)  ^{T}$ is a vector-valued solution of the linear system%
\begin{equation}
\left(
\begin{array}
[c]{cc}%
\overline{\partial} & 0\\
0 & \partial
\end{array}
\right)  \Phi=\frac{1}{2}\left(
\begin{array}
[c]{cc}%
0 & u\\
u & 0
\end{array}
\right)  \Phi\label{eq:ddbar}%
\end{equation}
A straightforward calculation shows that the function%
\[
\widetilde{\psi}=\Phi_{1}+\Phi_{2}%
\]
solves the zero-energy Schr\"{o}dinger equation

\begin{equation}
\left(  -\Delta+q\right)  \widetilde{\psi}=0 \label{eq:Szero}%
\end{equation}
for $q=2\partial u+\left\vert u\right\vert ^{2}$.

Recall that matrix-valued solutions of (\ref{eq:ddbar}) are related to the
solutions $\mu$ of (\ref{eq:mu.dbar})\ by
\[
\left(
\begin{array}
[c]{c}%
\mu_{1}\\
\mu_{2}%
\end{array}
\right)  =\left(
\begin{array}
[c]{c}%
\Phi_{1}\\
\overline{\Phi_{2}}%
\end{array}
\right)  e^{-kz}%
\]
so that%
\begin{equation}
\Phi_{11}+\Phi_{21}=e^{kz}\mu_{1}(z,k)+e^{\overline{k}\overline{z}}%
\overline{\mu_{2}(z,k)}\label{eq:mu.to.phi}%
\end{equation}
solves (\ref{eq:Szero}). To compute its asymptotic behavior, using $\left(
\mu_{1},\mu_{2}\right)  \rightarrow(1,0)$ as $\left\vert z\right\vert
\rightarrow\infty$ we conclude that $e^{-kz}\widetilde{\psi}(z,k)\rightarrow1$
as $\left\vert z\right\vert \rightarrow\infty$. Hence, denoting by $\psi$ the
solution of the problem (\ref{eq:Szero}) with $\zeta=(k,ik)$ for
$k\in\mathbb{C}$, we have%
\begin{align}
\psi(z,k) &  =\widetilde{\psi}(z,ik)\label{eq:mu.to.psi}\\
&  =e^{ikz}\mu_{1}(z,ik)+e^{-i\overline{k}\overline{z}}\overline{\mu
_{2}(z,ik)}\nonumber
\end{align}
so%
\[
m(z,k)=\mu_{1}(z,k)+e^{-i(kz+\overline{k}\overline{z})}\overline{\mu
_{2}(z,ik)}.
\]
Now, we can prove:

\begin{lemma}
\label{lemma:R.to.T}Let $u\in\mathcal{C}_{0}^{\infty}(\mathbb{R}^{2})$ with
$\partial u=\overline{\partial u}$, suppose $\int u(z)~dA(z)=0$, and let
$q=2\partial u+\left\vert u\right\vert ^{2}$. Then%
\begin{equation}
\left(  \mathcal{T}q\right)  (k)=-2\pi i\overline{k}~\overline{(\mathcal{R}%
u)(ik)}. \label{eq:R.to.T}%
\end{equation}

\end{lemma}

\begin{proof}
We compute%
\begin{align*}
\left(  \mathcal{T}q\right)  (k)  &  =\int q(z)e^{i\overline{k}\overline{z}%
}\psi(z,k)~dA(z)\\
&  =\int2(\overline{\partial u})(z)e^{i\overline{(k}\overline{z}+kz)}\mu
_{1}(z,ik)~dA(z)\\
&  +\int2(\partial u)(z)\overline{\mu_{2}(z,ik)}~dA(z)\\
&  +\int\left\vert u(z)\right\vert ^{2}\left(  e^{i\overline{(k}\overline
{z}+kz)}\mu_{1}(z,ik)+\overline{\mu_{2}(z,ik)}\right)  ~dA(z)\\
&  =I_{1}+I_{2}+I_{3}%
\end{align*}
where in the first right-hand term we used $\partial u=\overline{\partial u}$.
We can integrate by parts in each of the first two right-hand terms and use
(\ref{eq:mu.dbar}) to obtain%
\begin{align*}
I_{1}  &  =-2i\overline{k}\int\overline{u(z)}e^{i\left(  kz+\overline
{k}\overline{z}\right)  }\mu_{1}(z,ik)~dA(z)-\int\left\vert u(z)\right\vert
^{2}\overline{\mu_{2}(z,ik)}~dA(z),\\
I_{2}  &  =-\int\left\vert u(z)\right\vert ^{2}e^{i\left(  kz+\overline
{k}\overline{z}\right)  }\mu_{1}(z,ik)~dA(z).
\end{align*}
Using the relation (\ref{eq:r.rep}), we recover (\ref{eq:R.to.T}).
\end{proof}

Next, we analyze the inverse scattering transform $\mathcal{Q}$ defined by
(\ref{eq:q.rep}). We will prove:

\begin{lemma}
\label{lemma:MI.to.Q}Let $u\in\mathcal{S}(\mathbb{R}^{2})$ with $\partial
u=\overline{\partial u}$, and suppose that $\int u(z)~dA(z)=0$. Let
$r=\mathcal{R}u$ and suppose that $t$ is given by (\ref{eq:R.to.T}). Then%
\[
(\mathcal{Q}t)(z)=2\left(  \partial u\right)  (z)+\left\vert u(z)\right\vert
^{2}.
\]

\end{lemma}

\begin{proof}
We compute from (\ref{eq:q.rep}), (\ref{eq:R.to.T}), and (\ref{eq:mu.to.psi})
that
\begin{align*}
\left(  \mathcal{Q}t\right)  (z) &  =\frac{2}{\pi}\overline{\partial}%
_{z}\left(  \int\overline{r(ik)}~e^{-i\left(  kz+\overline{k}\overline
{z}\right)  }\overline{\mu_{1}(z,ik)}~dA(k)\right)  \\
&  +\frac{2}{\pi}\overline{\partial}_{z}\left(  \int\overline{r(ik)}~\mu
_{2}(z,ik)~dA(k)\right)  \\
&  =T_{1}+T_{2}%
\end{align*}
Changing variables to $\zeta=ik$ in $T_{1}$ we recover%
\begin{align*}
T_{1} &  =\frac{2}{\pi}\overline{\partial}_{z}\left(  \int\overline{r(\zeta
)}~e^{\overline{\zeta}\overline{z}-\zeta z}~\overline{\mu_{1}(z,\zeta
)}~dA(\zeta)\right)  \\
&  =2\left(  \overline{\partial u}\right)  (z)\\
&  =2\left(  \partial u\right)  (z)
\end{align*}
where we have used (\ref{eq:u.rep}). Using (\ref{eq:mu.dbar}) in $T_{2}$ we
have%
\begin{align*}
T_{2} &  =\frac{1}{\pi}\int\overline{r(ik)}~u(z)~e^{-i\left(  kz+\overline
{k}\overline{z}\right)  }\overline{\mu_{1}(z,ik)}~dA(k)\\
&  =\frac{1}{\pi}u(z)\int r(\zeta)~e^{\overline{\zeta}\overline{z}-\zeta
z}~\overline{\mu_{1}(z,\zeta)}~dA(\zeta)\\
&  =\left\vert u(z)\right\vert ^{2}.
\end{align*}
Combining these computations gives the desired result.
\end{proof}

\begin{proof}
[Proof of Theorem \ref{thm:PP.and.LMSS}]For $u_{0}$ satisfying the hypotheses
and $q=2\partial u_{0}+\left\vert u_{0}\right\vert ^{2}$, we have by Lemma
\ref{lemma:R.to.T} that
\[
\left(  \mathcal{T}q_{0}\right)  (k)=-2\pi i~\overline{k}~\overline{r(ik)}%
\]
where $r=\mathcal{R}(u_{0})$, and hence%
\[
e^{-it\left(  k^{3}+\overline{k}^{3}\right)  }\left(  \mathcal{T}q_{0}\right)
(k)=-2\pi i~\overline{k}~\overline{\left(  e^{t\left(  \left(  \overline
{\diamond}\right)  ^{3}-\left(  \diamond\right)  ^{3}\right)  }r(\diamond
)\right)  (ik)}.
\]
We can now apply Lemma \ref{lemma:MI.to.Q} to conclude that%
\[
\mathcal{Q}\left(  e^{-it\left(  \left(  \diamond\right)  ^{3}+\left(
\overline{\diamond}\right)  ^{3}\right)  }\left(  \mathcal{T}q_{0}\right)
(\diamond)\right)  =\mathcal{MI}\left(  e^{t\left(  \left(  \overline
{\diamond}\right)  ^{3}-\left(  \diamond\right)  ^{3}\right)  }r(\diamond
)\right)
\]
as claimed.
\end{proof}

\appendix

\section{Schwarz Class Inverse Scattering for the mNV Equation}

\label{app:mNV}

In this appendix we develop the Schwarz class inverse theory for the
mNV\ equation, using freely the results and notation of \cite{P} with one
exception: we denote the potential by $u$ rather than $q$. Our main result is:

\begin{theorem}
\label{thm:mNV.Schwartz}Suppose that $u_{0}\in\mathcal{S}(\mathbb{R}^{2})$,
and let $\mathcal{R}$ and $\mathcal{I}$ be the scattering maps defined
respectively by (\ref{eq:r.rep})\ and (\ref{eq:u.rep}). Finally, define%
\[
u(t)=\mathcal{I}\left(  e^{t\left(  \left(  \diamond\right)  ^{3}-\left(
\overline{\diamond}\right)  ^{3}\right)  }(\mathcal{R}u_{0})(\diamond)\right)
.
\]
Then $u(t)$ is a classical solution of the modified Novikov-Veselov equation
(\ref{eq:mNV}).
\end{theorem}

The proof follows the method of Beals-Coifman \cite{BC:1985,BC:1989,BC:1990}
and Sung \cite{Sung:1994} but necessitates some long computations.

\subsection{Scattering Solutions and Tangent Maps}

First we recall the solutions $\nu$ and $\widetilde{\nu}$ of the
$\overline{\partial}$ problem with $\overline{\partial}$-data determined by
the time-dependent coefficient $r$ and the formulas from \cite{P} for the
tangent maps.

We recall that $\nu=\left(  \nu_{1},\nu_{2}\right)  ^{T}$ is the unique
solution of the $\overline{\partial}$ problem%
\begin{align}
\overline{\partial}_{k}\nu_{1} &  =\frac{1}{2}e_{k}\overline{r}\overline
{\nu_{2}},\label{eq:nu.dbar.bis}\\
\overline{\partial}_{k}\nu_{2} &  =\frac{1}{2}e_{k}\overline{r}\overline
{\nu_{1}},\nonumber\\
\lim_{\left\vert k\right\vert \rightarrow\infty}\nu(z,k) &  =(1,0),\nonumber
\end{align}
where $r=\mathcal{R}(u)$. Here%
\[
e_{k}(z)=e^{\overline{k}\overline{z}-kz}.
\]
The function $\nu^{\#}=\left(  \nu_{1}^{\#},\nu_{2}^{\#}\right)  $ solves the
same problem but for $u^{\#}(~\cdot~)=-\overline{u}(-\cdot~~)$ and
$r^{\#}=\mathcal{R}(u^{\#})=-\overline{r}$ (see Lemma B.1 in \cite{P}). Thus%
\begin{align}
\overline{\partial}_{k}\nu_{1}^{\#} &  =-\frac{1}{2}e_{k}r\overline{\nu
_{2}^{\#}},\label{eq:tildenu.dbar}\\
\overline{\partial}_{k}\nu_{2}^{\#} &  =-\frac{1}{2}e_{k}\overline{r}%
\overline{\nu_{1}^{\#}},\nonumber\\
\lim_{\left\vert k\right\vert \rightarrow\infty}\nu^{\#}(z,k) &
=(1,0).\nonumber
\end{align}

The tangent map formula gives an expression for $u$ if $u=\mathcal{R}(r)$
where $r$ is a $C^{1}$-curve in $\mathcal{S}(\mathbb{R}^{2})$. Assuming the
law of evolution%
\[
\dot{r}=\left(  \overline{k}^{3}-k^{3}\right)  r
\]
and following the calculations in Appendix B of \cite{P}, we find that
\begin{equation}
u=2i(I_{1}+\overline{I_{2}}) \label{eq:qdot}%
\end{equation}
where%
\begin{align}
I_{1}  &  =\frac{1}{\pi}\int k^{3}\overline{\partial}_{k}\left[  \nu_{2}%
^{\#}(-z,k)\nu_{1}(z,k)\right]  ~dA(k),\label{eq:I1}\\
I_{2}  &  =-\frac{1}{\pi}\int k^{3}\overline{\partial}_{k}\left[  \nu_{1}%
^{\#}(-z,k)\nu_{2}(z,k)\right]  ~dA(k). \label{eq:I2}%
\end{align}

As in Appendix B of \cite{P}, we evaluate these integrals using the following
fact: if $g$ is a $\mathcal{C}^{\infty}$ function with asymptotic expansion%
\begin{equation}
g(k,\overline{k})\sim1+\sum_{\ell\geq0}\frac{g_{\ell}}{k^{\ell+1}}
\label{eq:k.asy}%
\end{equation}
as $\left\vert k\right\vert \rightarrow\infty$ then%
\begin{equation}
\lim_{R\rightarrow\infty}\left(  -\frac{1}{\pi}\int_{\left\vert k\right\vert
\leq R}k^{n}(\overline{\partial_{k}}g)(k)~dA(k)\right)  =g_{n}.
\label{eq:residue}%
\end{equation}
Using (\ref{eq:residue}) we get (noting the $-$ sign in (\ref{eq:I2}))%
\begin{align*}
I_{1}  &  =2~\left[  \nu_{1}(z,\diamond)\nu_{2}^{\#}(-z,\diamond)\right]
_{3}\\
\overline{I_{2}}  &  =2~\left[  \nu_{2}(z,\diamond)\nu_{1}^{\#}(-z,\diamond
)\right]  _{3}%
\end{align*}
so that%
\begin{equation}
\dot{u}=2\left\{  \left[  \nu_{1}(z,\diamond)\nu_{2}^{\#}(-z,\diamond)\right]
_{3}+\overline{\left[  \nu_{2}(z,\diamond)\nu_{1}^{\#}(-z,\diamond)\right]
_{3}}\right\}  \label{eq:udot.final}%
\end{equation}
Here $\left[  ~\diamond~\right]  _{n}$ denotes the coefficient of $k^{-n-1}$
in an asymptotic expansion of the form (\ref{eq:k.asy}). \ The
formulas\bigskip%
\begin{align*}
\left[  \nu_{1}(z,\diamond)\nu_{2}^{\#}(-z,\diamond)\right]  _{n}  &  =\left(
\nu_{n}^{\#}\right)  _{21}+\sum_{j=0}^{n-1}\left(  \nu_{n-j-1}^{\#}\right)
_{21}\left(  \nu_{j}\right)  _{11}\\
\left[  \nu_{2}(z,\diamond)\nu_{1}^{\#}(-z,\diamond)\right]  _{n}  &  =\left(
\nu_{n}\right)  _{12}+\sum_{j=0}^{n-1}\left(  \nu_{n-1-j}\right)  _{12}\left(
\nu_{j}^{\#}\right)  _{22}%
\end{align*}
will be used in concert with the residue formulae below to obtain the equation
of motion.

\subsection{Expansion Coefficients for $\nu$}

Following the method of Appendix C in \cite{P}, we can compute the additional
coefficients in the asymptotic expansion%
\begin{equation}
\nu\sim\left(  1,0\right)  +\sum_{\ell\geq0}k^{-(\ell+1)}\nu^{(\ell)}
\label{eq:nu.asy}%
\end{equation}
needed to compute $\dot{u}$ from the formula (\ref{eq:udot.final}). Let us set
$\nu^{(\ell)}=\left(  \nu_{1,\ell},\nu_{2,\ell}\right)  ^{T}$. We recall from
\cite{P} the `initial data'%
\begin{equation}
\nu_{1,0}=\frac{1}{4}\overline{\partial}^{-1}\left(  \left\vert u\right\vert
^{2}\right)  ,~~\nu_{2,0}=\frac{1}{2}\overline{u} \label{eq:nu.start}%
\end{equation}
and the recurrence relations%
\begin{align*}
\nu_{2,\ell}  &  =\frac{1}{2}\overline{u}\nu_{1,\ell-1}-\partial\nu_{2,\ell
-1},\\
\nu_{1,\ell}  &  =\frac{1}{2}P\left(  u\nu_{2,\ell}\right)  .
\end{align*}
The following formulas are a straightforward consequence.

\bigskip

$\ell=0$:%
\begin{align}
\nu_{1,0}  &  =\frac{1}{4}\overline{\partial}^{-1}\left(  \left\vert
u\right\vert ^{2}\right) \label{eq:nu.0.12}\\
\nu_{2,0}  &  =\frac{1}{2}\overline{u} \label{eq:nu.0.21}%
\end{align}

\bigskip

$\ell=1$:%
\begin{align}
\nu_{1,1}  &  =\frac{1}{16}\overline{\partial}^{-1}\left(  \left\vert
u\right\vert ^{2}\overline{\partial}^{-1}\left(  \left\vert u\right\vert
^{2}\right)  \right)  -\frac{1}{4}\overline{\partial}^{-1}\left(
u\partial\overline{u}\right) \label{eq:nu.1.12}\\
\nu_{2,1}  &  =\frac{1}{8}\overline{u}\overline{\partial}^{-1}\left(
\left\vert u\right\vert ^{2}\right)  -\frac{1}{2}\partial\overline{u}
\label{eq:nu.1.21}%
\end{align}

\bigskip

$\ell=2$:%
\begin{align}
\nu_{1,2}  &  =\frac{1}{64}\overline{\partial}^{-1}\left(  \left\vert
u\right\vert ^{2}\overline{\partial}^{-1}\left(  \left\vert u\right\vert
^{2}\overline{\partial}^{-1}\left(  \left\vert u\right\vert ^{2}\right)
\right)  \right) \label{eq:nu.2.11}\\
&  -\frac{1}{16}\left\{  \overline{\partial}^{-1}\left(  u\partial\left(
\overline{u}\overline{\partial}^{-1}\left(  \left\vert u\right\vert
^{2}\right)  \right)  \right)  +\overline{\partial}^{-1}\left(  \left\vert
u\right\vert ^{2}\overline{\partial}^{-1}\left(  u\partial\overline{u}\right)
\right)  \right\} \nonumber\\
&  +\frac{1}{4}\overline{\partial}^{-1}\left(  u\partial^{2}\overline
{u}\right) \nonumber\\
\nu_{2,2,}  &  =\frac{1}{32}\overline{u}\overline{\partial}^{-1}\left(
\left\vert u\right\vert ^{2}\overline{\partial}^{-1}\left(  \left\vert
u\right\vert ^{2}\right)  \right) \label{eq:nu.2.21}\\
&  -\frac{1}{8}\left\{  \partial\left(  \overline{u}\overline{\partial}%
^{-1}\left(  \left\vert u\right\vert ^{2}\right)  \right)  +\overline
{u}\overline{\partial}^{-1}\left(  u\partial\overline{u}\right)  \right\}
\nonumber\\
&  +\frac{1}{2}\partial^{2}\overline{u}\nonumber
\end{align}

\bigskip

$\ell=3$:%

\begin{align}
\nu_{2,3}  &  =\frac{1}{128}\overline{u}\overline{\partial}^{-1}\left(
\left\vert u\right\vert ^{2}\overline{\partial}^{-1}\left(  \left\vert
u\right\vert ^{2}\overline{\partial}^{-1}\left(  \left\vert u\right\vert
^{2}\right)  \right)  \right) \label{eq:nu.3.12}\\
&  -\frac{1}{32}\left\{  \overline{u}\overline{\partial}^{-1}\left(
u\partial\left(  \overline{u}\overline{\partial}^{-1}\left(  \left\vert
u\right\vert ^{2}\right)  \right)  \right)  +\overline{u}\overline{\partial
}^{-1}\left(  \left\vert u\right\vert ^{2}\overline{\partial}^{-1}\left(
u\partial\overline{u}\right)  \right)  \right. \nonumber\\
&  \left.  +\partial\left(  \overline{u}\overline{\partial}^{-1}\left(
\left\vert u\right\vert ^{2}\overline{\partial}^{-1}\left(  \left\vert
u\right\vert ^{2}\right)  \right)  \right)  \right\} \nonumber\\
&  +\frac{1}{8}\left\{  \overline{u}\overline{\partial}^{-1}\left(
u\partial^{2}\overline{u}\right)  +\partial^{2}\left(  \overline{u}%
\overline{\partial}^{-1}\left(  \left\vert u\right\vert ^{2}\right)  \right)
+\partial\left(  \overline{u}\overline{\partial}^{-1}\left(  u\partial
\overline{u}\right)  \right)  \right\} \nonumber\\
&  -\frac{1}{2}\partial^{3}\overline{u}\nonumber
\end{align}

\subsection{Expansion Coefficients for $\nu^{\#}$}

The solution $\nu^{\#}$ corresponds to the potential $-\overline{u}(-z)$. To
compute the corresponding residues for $\nu^{\#}\left(  -z,k\right)  $ we
therefore make the following substitutions in the formulas above:%
\begin{align*}
\overline{\partial}^{-1}  &  \rightarrow-\overline{\partial}^{-1}\\
\partial &  \rightarrow-\partial\\
u  &  \rightarrow-\lambda\overline{u}\\
\overline{u}  &  \rightarrow-\lambda u
\end{align*}
Thus the overall sign change is $\left(  -1\right)  ^{n_{u}+n_{\partial}}$
where $n_{u}$ is the number of factors of $u$ and $\overline{u}$, while
$n_{\partial}$ is the number of factors of $\partial$ and $\overline{\partial
}^{-1}$. There is also an overall factor of $\left(  \lambda\right)  ^{n_{u}}$
i.e. $\lambda$ if $n_{u}$ is odd, or $1$ if $n_{u}$ is even. Applying these
rules we obtain:

\bigskip

$\ell=0$:%
\begin{align}
\nu_{1,0}^{\#}  &  =-\frac{1}{4}\overline{\partial}^{-1}\left(  \left\vert
u\right\vert ^{2}\right) \label{eq:nu.tilde.0.12}\\
\nu_{2,0}^{\#}  &  =-\frac{1}{2}u \label{eq:nu.0.tilde.21}%
\end{align}

\bigskip

$\ell=1$:%
\begin{align}
\nu_{1,1}^{\#}  &  =\frac{1}{16}\overline{\partial}^{-1}\left(  \left\vert
u\right\vert ^{2}\overline{\partial}^{-1}\left(  \left\vert u\right\vert
^{2}\right)  \right)  -\frac{1}{4}\overline{\partial}^{-1}\left(  \overline
{u}\partial u\right) \label{eq:nu.tilde.1.12}\\
\nu_{2,1}^{\#}  &  =\frac{1}{8}u\overline{\partial}^{-1}\left(  \left\vert
u\right\vert ^{2}\right)  -\frac{1}{2}\partial u \label{eq:nu.tilde.1.21}%
\end{align}

\bigskip

$\ell=2$:%
\begin{align}
\nu_{1,2}^{\#}  &  =-\frac{1}{64}\overline{\partial}^{-1}\left(  \left\vert
u\right\vert ^{2}\overline{\partial}^{-1}\left(  \left\vert u\right\vert
^{2}\overline{\partial}^{-1}\left(  \left\vert u\right\vert ^{2}\right)
\right)  \right) \label{eq:nu.tilde.2.11}\\
&  +\frac{1}{16}\left\{  \overline{\partial}^{-1}\left(  \overline{u}%
\partial\left(  u\overline{\partial}^{-1}\left(  \left\vert u\right\vert
^{2}\right)  \right)  \right)  +\overline{\partial}^{-1}\left(  \left\vert
u\right\vert ^{2}\overline{\partial}^{-1}\left(  \overline{u}\partial
u\right)  \right)  \right\} \nonumber\\
&  -\frac{1}{4}\overline{\partial}^{-1}\left(  \overline{u}\partial
^{2}u\right) \nonumber\\
\nu_{2,2}^{\#}  &  =-\frac{1}{32}u\overline{\partial}^{-1}\left(  \left\vert
u\right\vert ^{2}\overline{\partial}^{-1}\left(  \left\vert u\right\vert
^{2}\right)  \right) \label{eq:nu.tilde.2.21}\\
&  +\frac{1}{8}\left\{  \partial\left(  u\overline{\partial}^{-1}\left(
\left\vert u\right\vert ^{2}\right)  \right)  +u\overline{\partial}%
^{-1}\left(  \overline{u}\partial u\right)  \right\} \nonumber\\
&  -\frac{1}{2}\partial^{2}u\nonumber
\end{align}

\bigskip

$\ell=3$:%

\begin{align}
\nu_{2,3}^{\#}  &  =\frac{1}{128}u\overline{\partial}^{-1}\left(  \left\vert
u\right\vert ^{2}\overline{\partial}^{-1}\left(  \left\vert u\right\vert
^{2}\overline{\partial}^{-1}\left(  \left\vert u\right\vert ^{2}\right)
\right)  \right) \label{eq:nu.tilde.3.12}\\
&  -\frac{1}{32}\left\{  u\overline{\partial}^{-1}\left(  \overline{u}%
\partial\left(  u\overline{\partial}^{-1}\left(  \left\vert u\right\vert
^{2}\right)  \right)  \right)  +u\overline{\partial}^{-1}\left(  \left\vert
u\right\vert ^{2}\overline{\partial}^{-1}\left(  \overline{u}\partial
u\right)  \right)  \right. \nonumber\\
&  \left.  +\partial\left(  u\overline{\partial}^{-1}\left(  \left\vert
u\right\vert ^{2}\overline{\partial}^{-1}\left(  \left\vert u\right\vert
^{2}\right)  \right)  \right)  \right\} \nonumber\\
&  +\frac{1}{8}\left\{  u\overline{\partial}^{-1}\left(  \overline{u}%
\partial^{2}u\right)  +\partial^{2}\left(  u\overline{\partial}^{-1}\left(
\left\vert u\right\vert ^{2}\right)  \right)  +\partial\left(  u\overline
{\partial}^{-1}\left(  \overline{u}\partial u\right)  \right)  \right\}
\nonumber\\
&  -\frac{1}{2}\partial^{3}u\nonumber
\end{align}

\subsection{Inverse Scattering Method for mNV}

We now compute the motion of the putative solution%
\[
u=\mathcal{I}r
\]
if the reflection coefficient evolves according to the law%
\begin{align*}
\dot{r} &  =-\left(  k^{3}-\overline{k}^{3}\right)  r\\
\left.  r\right\vert _{t=0} &  =\mathcal{R}u_{0}%
\end{align*}
From (\ref{eq:udot.final}) it is clear that we must compute $\left[  \nu
_{1}(z,\diamond)\nu_{2}^{\#}(-z,\diamond)\right]  _{3}$ and $\left[  \nu
_{2}(z,\diamond)\nu_{1}^{\#}(-z,\diamond)\right]  _{3}.$

First, we have%
\begin{equation}
\left[  \nu_{1}(z,\diamond)\nu_{2}^{\#}(-z,\diamond)\right]  _{3}=\nu
_{2,3}^{\#}+\nu_{2,2}^{\#}\nu_{1,0}+\nu_{2,1}^{\#}\nu_{1,1}+\nu_{2,0}^{\#}%
\nu_{1,2}. \label{eq:t1.sum}%
\end{equation}
From the formulas above we have%
\begin{align}
\nu_{2,2}^{\#}\nu_{1,0}  &  =-\frac{1}{128}u\overline{\partial}^{-1}\left(
\left\vert u\right\vert ^{2}\overline{\partial}^{-1}\left(  \left\vert
u\right\vert ^{2}\right)  \right)  \cdot\left(  \overline{\partial}%
^{-1}\left\vert u\right\vert ^{2}\right) \label{eq:t1.2}\\
&  +\frac{1}{32}\left\{  \partial\left(  u\overline{\partial}^{-1}\left(
\left\vert u\right\vert ^{2}\right)  \right)  \cdot\left(  \overline{\partial
}^{-1}\left(  \left\vert u\right\vert ^{2}\right)  \right)  +u\overline
{\partial}^{-1}\left(  \overline{u}\partial u\right)  \cdot\left(
\overline{\partial}^{-1}\left(  \left\vert u\right\vert ^{2}\right)  \right)
\right\} \nonumber\\
&  -\frac{1}{8}\partial^{2}u\cdot\overline{\partial}^{-1}\left(  \left\vert
u\right\vert ^{2}\right)  ,\nonumber
\end{align}%
\begin{align}
\nu_{2,1}^{\#}\nu_{1,1}  &  =\frac{1}{128}u\overline{\partial}^{-1}\left(
\left\vert u\right\vert ^{2}\right)  \cdot\overline{\partial}^{-1}\left(
\left\vert u\right\vert ^{2}\overline{\partial}^{-1}\left(  \left\vert
u\right\vert ^{2}\right)  \right) \label{eq:t1.3}\\
&  -\frac{1}{32}\left\{  u\overline{\partial}^{-1}\left(  \left\vert
u\right\vert ^{2}\right)  \cdot\overline{\partial}^{-1}\left(  u\partial
\overline{u}\right)  +\partial u\cdot\left(  \overline{\partial}^{-1}\left(
\left\vert u\right\vert ^{2}\overline{\partial}^{-1}\left(  \left\vert
u\right\vert ^{2}\right)  \right)  \right)  \right\} \nonumber\\
&  +\frac{1}{8}\partial u\cdot\overline{\partial}^{-1}\left(  u\partial
\overline{u}\right)  ,\nonumber
\end{align}
and%
\begin{align}
\nu_{2,0}^{\#}\nu_{1,2}  &  =-\frac{1}{128}u\overline{\partial}^{-1}\left(
\left\vert u\right\vert ^{2}\overline{\partial}^{-1}\left(  \left\vert
u\right\vert ^{2}\overline{\partial}^{-1}\left(  \left\vert u\right\vert
^{2}\right)  \right)  \right) \label{eq:t1.4}\\
&  +\frac{1}{32}\left\{  u\overline{\partial}^{-1}\left(  u\partial\left(
\overline{u}\overline{\partial}^{-1}\left(  \left\vert u\right\vert
^{2}\right)  \right)  \right)  +u\overline{\partial}^{-1}\left(  \left\vert
u\right\vert ^{2}\overline{\partial}^{-1}\left(  u\partial\overline{u}\right)
\right)  \right\} \nonumber\\
&  -\frac{1}{8}u\overline{\partial}^{-1}\left(  u\partial^{2}\overline
{u}\right)  .\nonumber
\end{align}
Using (\ref{eq:nu.tilde.3.12}) and (\ref{eq:t1.2})-(\ref{eq:t1.4}) in
(\ref{eq:t1.sum}) we see that seventh-order terms cancel, while fifth-order
terms sum to zero, as may be shown using the identity%
\begin{equation}
\overline{\partial}^{-1}f\cdot\overline{\partial}^{-1}g=\overline{\partial
}^{-1}\left(  f\overline{\partial}^{-1}g+g\overline{\partial}^{-1}f\right)  ,
\label{eq:invprod}%
\end{equation}
while third-order terms may be simplified using the same identity with $f=g$.
The result is
\begin{align}
\left[  \nu_{11}(z,\diamond)\widetilde{\nu}_{21}(-z,\diamond)\right]  _{3}  &
=\frac{3}{8}\left[  \left(  \partial u\right)  \cdot\left(  \overline
{\partial}^{-1}\left(  \partial\left(  \left\vert u\right\vert ^{2}\right)
\right)  \right)  \right]  +\frac{3}{8}\left[  u\overline{\partial}%
^{-1}\left(  \overline{u}\overline{\partial}u\right)  \right]
\label{eq:res3.1}\\
&  -\frac{1}{2}\partial^{3}u\nonumber
\end{align}

Next, we compute
\begin{equation}
\left[  \nu_{2}(z,\diamond)\nu_{1}^{\#}(-z,\diamond)\right]  _{3}=\nu
_{2,3}+\nu_{2,2}\nu_{1,0}^{\#}+\nu_{2,1}\nu_{1,1}^{\#}+\nu_{2,0}\nu_{1,2}%
^{\#}. \label{eq:t2.sum}%
\end{equation}
From the formulas above we have
\begin{align}
\nu_{2,2}\nu_{1,0}^{\#}  &  =-\frac{\lambda}{128}\overline{u}\overline
{\partial}^{-1}\left(  \left\vert u\right\vert ^{2}\overline{\partial}%
^{-1}\left(  \left\vert u\right\vert ^{2}\right)  \right)  \cdot
\overline{\partial}^{-1}\left(  \left\vert u\right\vert ^{2}\right)
\label{eq:t2.2}\\
&  +\frac{1}{32}\left\{  \partial\left(  \overline{u}\overline{\partial}%
^{-1}\left(  \left\vert u\right\vert ^{2}\right)  \right)  \cdot
\overline{\partial}^{-1}\left(  \left\vert u\right\vert ^{2}\right)
+\overline{u}\overline{\partial}^{-1}\left(  u\partial\overline{u}\right)
\cdot\left(  \overline{\partial}^{-1}\left(  \left\vert u\right\vert
^{2}\right)  \right)  \right\} \nonumber\\
&  -\frac{\lambda}{8}\partial^{2}\overline{u}\cdot\overline{\partial}%
^{-1}\left(  \left\vert u\right\vert ^{2}\right)  ,\nonumber
\end{align}%
\begin{align}
\nu_{2,1}\nu_{1,1}^{\#}  &  =\frac{1}{128}\overline{u}\overline{\partial}%
^{-1}\left(  \left\vert u\right\vert ^{2}\right)  \cdot\overline{\partial
}^{-1}\left(  \left\vert u\right\vert ^{2}\overline{\partial}^{-1}\left(
\left\vert u\right\vert ^{2}\right)  \right) \label{eq:t2.3}\\
&  -\frac{1}{32}\left\{  \overline{u}\overline{\partial}^{-1}\left(
\left\vert u\right\vert ^{2}\right)  \cdot\overline{\partial}^{-1}\left(
\overline{u}\partial u\right)  +\partial\overline{u}\cdot\overline{\partial
}^{-1}\left(  \left\vert u\right\vert ^{2}\overline{\partial}^{-1}\left(
\left\vert u\right\vert ^{2}\right)  \right)  \right\} \nonumber\\
&  +\frac{1}{8}\partial\overline{u}\cdot\overline{\partial}^{-1}\left(
\overline{u}\partial u\right)  ,\nonumber
\end{align}
and%
\begin{align}
\nu_{2,0}\nu_{1,2}^{\#}  &  =-\frac{1}{128}\overline{u}\overline{\partial
}^{-1}\left(  \left\vert u\right\vert ^{2}\overline{\partial}^{-1}\left(
\left\vert u\right\vert ^{2}\overline{\partial}^{-1}\left(  \left\vert
u\right\vert ^{2}\right)  \right)  \right) \label{eq:t2.4}\\
&  +\frac{1}{32}\left\{  \overline{u}\overline{\partial}^{-1}\left(
\overline{u}\partial\left(  u\overline{\partial}^{-1}\left(  \left\vert
u\right\vert ^{2}\right)  \right)  \right)  +\overline{u}\overline{\partial
}^{-1}\left(  \left\vert u\right\vert ^{2}\overline{\partial}^{-1}\left(
\overline{u}\partial u\right)  \right)  \right\} \nonumber\\
&  -\frac{1}{8}\overline{u}\overline{\partial}^{-1}\left(  \overline
{u}\partial^{2}u\right)  .\nonumber
\end{align}
Using (\ref{eq:nu.3.12}) and (\ref{eq:t2.2})-(\ref{eq:t2.4}) in
(\ref{eq:t2.sum}), noting the cancellation of fifth-order terms, we obtain%
\begin{align}
\left[  \nu_{2}(z,\diamond)\nu_{1}^{\#}(-z,\diamond)\right]  _{3}  &
=\frac{3}{8}\left[  \overline{u}\overline{\partial}^{-1}\left(  \partial
\left(  u\partial\overline{u}\right)  \right)  \right]  +\frac{3}{8}\left(
\partial\overline{u}\right)  \cdot\partial\overline{\partial}^{-1}\left(
\left\vert u\right\vert ^{2}\right) \label{eq:res3.2.pre}\\
&  -\frac{1}{2}\partial^{3}\overline{u}\nonumber
\end{align}
or upon complex conjugation%
\begin{align}
\overline{\left[  \nu_{2}(z,\diamond)\nu_{1}^{\#}(-z,\diamond)\right]  _{3}}
&  =\frac{3}{8}u\partial^{-1}\left(  \overline{\partial}\left(  \overline
{u}\overline{\partial}u\right)  \right)  +\frac{3}{8}\left(  \overline
{\partial}u\right)  \cdot\partial^{-1}\left(  \overline{\partial}\left(
\left\vert u\right\vert ^{2}\right)  \right) \label{eq:res3.2}\\
&  -\frac{1}{2}\overline{\partial}^{3}u\nonumber
\end{align}

Using these equations in (\ref{eq:udot.final}), we obtain the mNV equation:%
\begin{align}
\frac{\partial u}{\partial t}  &  =-\partial^{3}u-\overline{\partial}%
^{3}u\label{eq:mNV.proved}\\
&  +\frac{3}{4}\left(  \partial\overline{u}\right)  \cdot\left(
\overline{\partial}\partial^{-1}\left(  \left\vert u\right\vert ^{2}\right)
\right)  +\frac{3}{4}\left(  \overline{\partial}u\right)  \cdot\left(
\overline{\partial}\partial^{-1}\left(  \left\vert u\right\vert ^{2}\right)
\right) \nonumber\\
&  +\frac{3}{4}\overline{u}\overline{\partial}\partial^{-1}\left(
\overline{u}\overline{\partial}u\right)  +\frac{3}{4}u\partial^{-1}\left(
\overline{\partial}\left(  \overline{u}\overline{\partial}u\right)  \right)
.\nonumber
\end{align}

\end{document}